\documentclass[preprint,3p]{elsarticle}
\usepackage{stmaryrd}
\usepackage{amsmath}
\usepackage{amssymb}
\usepackage{tabularx}
\input psfig.sty

\newtheorem{proof}{Proof}[section]
\newtheorem{theorem}{Theorem}[section]
\newtheorem{lem}{Lemma}[section]

\newcommand{\bx}{{\bf x}}
\newcommand{\bn}{{\bf n}}

\newcommand{\vep}{\varepsilon}

\numberwithin{equation}{section}

\begin{document}

\begin{frontmatter}

\title{Numerical Study of Quantized Vortex Interaction
in Complex Ginzburg--Landau Equation on Bounded Domains}
\author[1]{Wei Jiang}
\ead{jiangwei1007@foxmail.com}
\address[1]{Beijing Computational Science Research Center, Beijing 100084, P. R. China}

\author[2]{Qinglin Tang}
\ead{tqltql2010@gmail.com}

\address[2]{Department of Mathematics and Center for
Computational Science and Engineering, National University of
Singapore, Singapore 119076, Singapore}


\begin{abstract}
In this paper, we study numerically quantized vortex
dynamics and their interaction in the two-dimensional
complex Ginzburg-Landau equation (CGLE) with a dimensionless
parameter $\vep>0$ on bounded domains under either Dirichlet or
homogeneous Neumann boundary condition. We begin with a
review of the reduced dynamical laws (RDLs) for time evolution of
quantized vortex centers in CGLE and show how to solve these
nonlinear ordinary differential equations numerically.
Then we present efficient and accurate numerical methods for solving
the CGLE on either a rectangular or a disk  domain under
either Dirichlet or homogeneous Neumann boundary condition.
Based on these efficient and accurate numerical methods
for CGLE and the RDLs, we explore rich and complicated
quantized vortex dynamics and interaction of CGLE with different $\vep$ and
under different initial physical setups, including single vortex, vortex pair,
vortex dipole and vortex lattice, compare them with those obtained
from the corresponding RDLs,
and identify the cases where the RDLs
agree qualitatively and/or quantitatively as well as
fail to agree with those from CGLE on vortex interaction.
Finally, we also obtain numerically different patterns of
the steady states for quantized vortex lattices
in the CGLE dynamics on bounded domains.
\end{abstract}



\end{frontmatter}


\section{Introduction}
\label{sec: intro}

Vortices are those waves that possess phase singularities
(topological defect) and rotational flows around the singular
points. They arise in many physical areas of different scale and
nature ranging from liquid crystals and superfluidity to
non-equilibrium patterns and cosmic strings \cite{DKT, PS}.
Quantized vortices in the two dimension are those particle-like vortices,
whose centers are the zero of the order parameter, possessing localized phase
singularity with the topological charge (also called as winding
number or index) being quantized. They have been widely observed in
many different physical systems, such as the liquid helium, type-II
superconductors, atomic gases and nonlinear optics \cite{A, BDZ, D,
FHL1, JN}. Quantized vortices are key signatures of the
superconductivity and superfluidity and their study is always one
of the most important and fundamental problems since they were
predicted by Lars Onsager in 1947 in connection with superfluid
Helium.

In this paper, we consider the vortex dynamics and
interactions in two dimensional complex Ginzburg--Landau equation (CGLE), which is
one of the most studied nonlinear equations in physics community \cite{AK}. It has attracted ever more attention,
because it can describe various phenomena ranging
from nonlinear waves to second-order phase transitions, from superconductivity, superfluidity and Bose-Einstein condensation
to liquid crystals and strings in field theory \cite{AK, FPR, GP, RB}. The specific form of  CGLE we study here reads as:
\begin{equation}
\label{cgle}
(\lambda_\vep+i\beta ) \partial_{t}\psi^{\vep} (\bx,t)=\Delta
\psi^{\vep}+\frac{1}{\vep^{2}}
(1-|\psi^{\vep}|^{2})\psi^{\vep},\qquad
\bx\in\mathcal{D},\quad t>0,
\end{equation}
with initial condition
\begin{equation}\label{ini_con}
\psi^{\vep}(\bx,0) = \psi^{\vep}_{0}(\bx), \qquad
  \bx\in\overline{\mathcal D },
\end{equation}
and under either Dirichlet boundary condition (BC)
\begin{equation}\label{dir}
\psi^{\vep}(\bx,t)=g(\bx)=e^{i\omega(\bx)},\qquad
\bx\in\partial\mathcal{D}, \quad t\ge0,
\end{equation}
or homogeneous Neumann BC
\begin{equation}\label{neu-bc}
\frac{\partial \psi^{\vep}(\bx,t)}{\partial \nu} =0,\qquad
\bx\in\partial\mathcal{D}, \quad t\ge0.
\end{equation}
where $\mathcal{D}\subset \mathbb{R}^{2}$ is a
bounded and simple connected domain in the paper, $t$ is time, $\bx=(x,y)\in \mathbb{R}^2$
is the Cartesian coordinate vector,
$\psi^\vep:=\psi^\vep(\bx,t)$ is
a complex-valued wave function (order parameter), $\omega$ is a given real-valued function,
$\psi_0^\vep$ and $g$ are given smooth and  complex-valued functions
satisfying  the compatibility condition $\psi_0^\vep(\bx)=g(\bx)$
for $\bx\in\partial\mathcal{D}$, $\nu=(\nu_1,\nu_2)$ and
$\nu_\perp=(-\nu_2,\nu_1)\in \mathbb{R}^2$ satisfying
$|\nu|=\sqrt{\nu_1^2+\nu_2^2}=1$ are the outward normal and tangent
vectors along $\partial \mathcal{D}$, respectively, $i=\sqrt{-1}$ is the unit
imaginary number, $0<\vep<1$ is a given dimensionless constant, and
$\lambda_\vep$, $\beta$ are two
positive constants. Actually,
the CGLE covers many different equations arise in various different  physical fields.
For example, when $\lambda_\vep\neq0$, $\beta=0$, it reduces to the Ginzburg-Landau equation (GLE)
for modelling superconductivity. When $\lambda_\vep=0$, $\beta=1$,  the CGLE
 collapses to the nonlinear Schr\"{o}dinger equation (NLSE)
for modelling Bose-Einstein Condensation or superfluidity.

Denote the Ginzburg-Landau (GL) functional (`energy') as \cite{CJ, JS, LX}
\begin{equation}
\label{glf}
\mathcal{E}^\vep(t):=\int_{\mathcal{D}}
\left[\frac{1}{2}|\nabla\psi^{\vep}|^2+
\frac{1}{4\vep^2}\left(1-|\psi^{
\vep}|^2\right)^2\right]d\bx
=\mathcal{E}_{\rm kin}^\vep(t)+\mathcal{E}_{\rm int}^\vep(t),
 \quad t\ge0,
\end{equation}
where the kinetic and interaction energies are defined as
\begin{equation*}
\mathcal{E}_{\rm kin}^\vep(t):=\frac{1}{2}\int_{\mathcal{D}}
|\nabla\psi^{\vep}|^2d\textbf{x},\quad \mathcal{E}_{\rm
int}^\vep(t):=\frac{1}{4\vep^2}\int_{\mathcal{D}}
\left(1-|\psi^{\vep}|^2\right)^2d\textbf{x},\quad t\ge0,
\end{equation*}
respectively. Then, it is easy to verify that the CGLE and GLE dissipate the energy, while the NLSE conserves
the energy at all the time.

During the last several decades, constructions and analysis about the vortex solutions
as well as studies of quantized vortex dynamics and interaction related with the CGLE
(\ref{cgle}) under different scalings have been extensively studied in the literatures.
For the whole space case , i.e., $\mathcal{D}=\mathbb{R}^{2}$,
Neu \cite{JN} studied dynamics and interaction of well-separated
quantized vortices for  GLE with $\lambda_\vep=1$ and NLSE under scaling $\vep=1$.
He found numerically that quantized vortices with winding number $m=\pm1$ are dynamically
stable, and respectively, $|m|>1$ dynamically unstable in the GLE dynamics. Moreover, he found
that vortices  behave like point vortices in ideal fluid.  Using asymptotic analysis, he derived the reduced
dynamical laws (RDLs) which are sets of ordinary differential equations (ODEs)
for governing the dynamics of the vortex centers to the leading order. Recently,
Neu's results were extented by Bethuel et al. to investigate the asymptotic behaviour of vortices
as $\vep\to 0$ in the GLE dynamics under the accelerating time scale $\lambda_\vep=\frac{1}{\ln\frac{1}{\vep}}$
\cite{BOS, BOS1, BOS2} and in the NLSE dynamics \cite{BJS}. The corresponding RDLs that govern the motion
of the limiting vortices have also be derived.

Inspired by Neu's work, many other papers have been dedicated to the study of the
vortex states and dynamics for the GLE and NLSE  with $0<\vep<1$ on a
bounded domain under different BCs. For the GLE case, Lin \cite{FHL1,FHL2,FHL3}
considered the the dynamics of vortices in the asymptotic limit  $\vep\to 0$
under various scales of $\lambda_\vep$ and with Dirichlet BC (\ref{dir}) or homogeneous Neumann (\ref{neu-bc}).
He derived the RDLs that govern the motion of these vortices and rigorously proved that
vortices move with velocities of order $|\ln\vep|^{-1}$ if $\lambda_\vep=1$.
Similar studies have also been conducted by E \cite{EW},  Jerrard et al. \cite{RLJHMH},
Jimbo et al. \cite{SJYM2, SJYM3} and  Sandier et al. \cite{ESSS}.
Unfortunately, all those RDLs are only valid up to the  first time that the vortices collide or exit
the domain and cannot describe the motion of multiple degree vortices.
Recently, Serfaty \cite{SE} extended the RDLs for the dynamics of the vortices after collisions.
For the NLSE case,  Mironescu \cite{M}  and Lin \cite{L} investigated stability of the
vortices in NLSE  with (\ref{dir}). Subsequently, Lin and Xin \cite{LX} studied the vortex
dynamics on a bounded domain with either Dirichlet or Neumann BC,
which was further investigated by Jerrard and Spirn \cite{JS}.
In addition, Colliander and Jerrard \cite{CJ, CJ1} studied the
vortex structures and dynamics on a torus or under periodic BC. In these studies,
RDLs were put forth to describe the asymptotic behaviour of the vortices
as $\vep\to 0$, which indicate that
to the leading order the vortices move according to the
Kirchhoff law in the bounded domain case. However, these RDLs cannot indicate radiation and/or sound propagations
created by highly co-rotating or overlapping vortices. In fact, it remains as
a very fascinating and fundamental open problem to understand the
vortex-sound interaction \cite{NBCT}, and how the sound waves modify
the motion of vortices \cite{FPR}.

For the CGLE, under scaling $\lambda_\vep\sim O(\frac{1}{\ln(1/\vep)})$,
Miot \cite{MIOT} studied the dynamics
of vortices asymptotically as $\vep\to 0$ in the whole plane case
and Kurzke et al. \cite{KMMS} investigated that in the bounded domain case, the corresponding RDLs
were derived to govern the motion of the limiting vortices in the
whole plane and/or the bounded domain, respectively.  Their results showed that the RDLs in the CGLE is actually
a hybrid of RDL for GLE and that for NLSE. More recently, Serfaty and Tice \cite{SI} studied the vortex dynamics
in a more complicated CGLE which involves  electromagnetic field and pinning effect.

On the numerical aspects, finite element methods were proposed to investigate numerical
solutions of GLE and related Ginzburg-Landau models
for modelling superconductivity \cite{QDMGJP, QDu0, KM, Aftalion2001, CD}.
Recently, by proposing efficient and accurate
numerical methods for solving the CGLE (\ref{cgle}) in the whole space, Zhang
et al. \cite{YZWBQD1,YZWBQD2} compared the dynamics of quantized
vortices from the RDLs obtained by Neu with those obtained
from the direct numerical simulation results from GLE and NLSE under
different parameters and initial setups. Very recently, The second author designed
some efficient and accurate numerical methods for studying vortex dynamics and interactions
in the GLE and/or NLSE on bounded domains with either Dirichlet
or Neumann BCs \cite{BT, BT1}. These numerical methods can be extended and applied
for studying the rich and complicated phenomena related to
vortex dynamics and interaction in the CGLE (\ref{cgle}) with either Dirichlet BC (\ref{dir}) or
homogeneous Neumann BC (\ref{neu-bc}) on bounded domains.
The main purpose of this paper is organised as:
(i). to present efficient and accurate
numerical methods for solving the RDLs
and the CGLE (\ref{cgle}) on bounded domains
under different BCs;
(ii). to understand numerically how the boundary condition and
geometry of the domain affect vortex dynamics and interction;
(iii). to study numerically vortex interaction in the CGLE dynamics and/or compare
them with those from the RDLs with different initial setups and
parameter regimes;
(iv). to identify cases where the reduced
dynamical laws agree qualitatively and/or quantitatively as well as
fail to agree with those from CGLE on vortex interaction.

The rest of the paper is organized as follows. In section 2, we briefly review the reduced
dynamical laws of vortex interaction under
the CGLE (\ref{cgle}) with either Dirithlet or
homogeneous Neumann BC and
present numerical methods to discretize them. In
section 3, efficient and accurate numerical methods are briefly outlined  for
solving the CGLE on bounded domains with different
BCs. In section 4 and section 5, ample numerical results are reported
for studying vortex dynamics and interaction of CGLE under Dirichlet BC and homogeneous Neumann BC.
Finally, some conclusions are drawn in section 6.

\section{The reduced dynamical laws and their discretization}
\label{sec: rdl}
The CGLE can be thought of as a hybird equation between the GLE and NLSE, and it has been proved that vortices in GLE dynamics
move with a velocity of the order of $\ln(1/\vep)$ if $\lambda_\vep=1$, Therefore, to obtain nontrivial
vortex dynamics, hereafter in this paper, we always choose
\begin{equation}
\label{lam_vep}
\lambda_\vep=\frac{\alpha}{\ln(1/\vep)},\qquad\qquad 0<\vep<1,
\end{equation}
where $\alpha$ is a positive number.
In this section, we review the RDLs for governing the dynamics of vortex centers in the CGLE (\ref{cgle})
with either Dirichlet or homogeneous Neumann BCs.

To simplify our discussion, for $j=1,\cdots,N$, hereafter we let $\bx^0_j(t)=(x^0_j,y^0_j)$ and
$\bx^\vep_j(t)=(x^\vep_j(t),y^\vep_j(t))$
be the location of the $M$ distinct and isolated vortex centers in the intial data $\psi^0$ (\ref{ini_con}) and
solution of the CGLE (\ref{cgle}) with initial condition (\ref{ini_con}) at time $t\ge0$, respectively.
By denoting
\[X^0:=(\bx_1^0,\bx_2^0,\ldots,\bx_M^0), \qquad X^\vep:=X^\vep(t)=(\bx^\vep_1(t),
\bx^\vep_2(t),\ldots,\bx^\vep_M(t)), \quad t\ge0,\]
then we have \cite{KMMS, CJ, FHL3}:
\begin{theorem}
\label{thm: rdl_gen}
As $\vep\rightarrow 0$,  for $j=1,\cdots,N$, the vortex center $\bx^\vep_j(t)$ will converge to
point $\bx_j(t)$ satisfying:
\begin{align}
&\label{eqn: rdl_gen}
(\alpha I+\beta n_{j} J) \frac{d\bx_{j}(t)}{dt}=-\nabla_{\bx_{j}}W(X), \quad 0\leq t<T,\\
& \bx_j(t=0)=\bx^{0}_j.
\end{align}
\end{theorem}
\noindent In equation (\ref{eqn: rdl_gen}), $T$ is the first time that either two vortex collide or any vortex exit the domain,
$n_j=+1$ or $-1$ is the winding number of the vortex,
 $X:=X(t)=(\bx_1(t), \bx_2(t),\ldots,\bx_M(t)),$
\[I=\begin{pmatrix}
   1 & 0 \\
   0 & 1 \\
 \end{pmatrix},\qquad J=\begin{pmatrix}
                        0 & -1 \\
                        1 & 0 \\
                      \end{pmatrix},\]
are the $2\times2$ identity and symplectic matrix, respectively. Moreover, the function $W(X)$ is the so called
renormalized energy defined as:
\begin{equation}
\label{re_ener}
W(X)=: W_{\rm cen}(X) +W_{\rm bc}(X),
\end{equation}
where $W_{\rm cen}$ is the renormalized energy associated to the $M$ vortex centers that defined as
\begin{equation}
\label{re_ener_cen}
W_{\rm cen}(X) = -\sum_{1\leq i\neq j\leq N}n_{i}n_{j}\ln|\bx_{i}-\bx_{j}|,
\end{equation}
and $W_{\rm bc}(X)$ is the renormalized energy involving the effect of the BC
(\ref{dir}) and/or (\ref{neu-bc}), which takes different formations in different cases.

\subsection{Under Dirichlet boundary condition}\label{sec: rdl_dir}

For the CGLE (\ref{cgle}) with initial condition (\ref{ini_con}) under Dirichlet BC (\ref{dir}),
it has been derived formally and rigorously \cite{FHL1, KMMS, LX1, CJ, BBH, SE} that
$W_{\rm bc}(X)=W_{\rm dbc}(X)$ in the renormalized energy $(\ref{re_ener})$ admits the form:
\begin{equation}
\label{re_ener_bc_dir}
W_{\rm dbc}(X) =: -\sum_{j=1}^{M}n_{j}R(\bx_{j};X)
+\int_{\partial\mathcal{D}}\left[R(\bx;X)+
\sum_{j=1}^{M}n_{j}\ln|\bx-\bx_{j}|\right]
\frac{\partial_{\nu_\perp} \omega(\bx)}{2\pi} \;ds,
\end{equation}
where, for any fixed $X\in \mathcal{D}^M$, $R(\bx;X)$ is a harmonic
function in $\bx$, i.e.,
\begin{equation}\label{harmR}
  \Delta R(\bx;X) = 0,\qquad \bx\in \mathcal{D},
\end{equation}
satisfying the following Neumann BC
\begin{equation}\label{harmRC}
\frac{\partial R(\bx;X)}{\partial \nu} =
\partial_{\nu_\perp} \omega(\bx) -\frac{\partial}{\partial
  \nu}\sum_{l=1}^{M}n_{l}\ln|\bx-\bx_{l}|, \qquad
  \bx\in \partial\mathcal{D}.
\end{equation}
Notice that to calculate $\nabla_{\bx_j} W(X)$, we need to calculate $\nabla_{\bx_j} R$, and
since for $j=1,\cdots,N$,  $\bx_j$  is implicitly included in $R(\bx,X)$ as a parameter,
hence it is difficult to calculate $\nabla_{\bx_j} R$ and thus difficult to solve
the RDL (\ref{eqn: rdl_gen}) with (\ref{re_ener})--(\ref{re_ener_bc_dir})
even numerically. However, by using an identity in \cite{BBH} (see Eq. (51) on page 84),
\[
\nabla_{\bx_j}\left[W(X)+ W_{\rm
dbc}(X)\right]=-2n_{j}\nabla_{\bx}\left[R(\bx;X)+\sum_{l=1\& l\ne
   j}^Mn_{l}\ln|\bx-\bx_{l}|\right]_{\bx=\bx_{j}}, \]
we have the following simplified equivalent form for (\ref{eqn: rdl_gen}).
\begin{lem}
 \label{lem: equi_dir_rdl1}
  For $1\le j\le M$ and $t>0$, system (\ref{eqn: rdl_gen}) can be simplified as
  \begin{equation}
\label{eqn: reduced1}
 (\alpha I+\beta m_{j} J) \frac{d}{dt}\bx_{j}(t)=2n_j\left[\nabla_\bx
   R\left(\bx;X\right)|_{\bx=\bx_{j}(t)}+\sum_{l=1\& l\ne
   j}^Mn_{l}\frac{\bx_{j}(t)-\bx_{l}(t)}{|\bx_{j}(t)
   -\bx_{l}(t)|^2}\right].
\end{equation}
\end{lem}
\noindent Moreover,  for any fixed $X\in \mathcal{D}^M$, by introducing function $H(\bx,X)$
and  $Q(\bx,X)$  that both are harmonic in $\bx$ satisfying respectively
the boundary condition \cite{RLJHMH, LX}:
\begin{align}
&\label{harmRC_tan}
\frac{\partial H(\bx;X)}{\partial \nu_\perp} =
\partial_{\nu_\perp} \omega(\bx) -\frac{\partial}{\partial
  \nu}\sum_{l=1}^{M}n_{l}\ln|\bx-\bx_{l}|, \qquad
  \bx\in \partial\mathcal{D},\\
 &\label{harmRC_dir}
 Q(\bx;X)= \omega(\bx)-\sum_{l=1}^{M}n_{l}\theta(\bx-\bx_{l}), \qquad
  \bx\in \partial\mathcal{D},
\end{align}
with the function $\theta: \ {\mathbb R}^2 \to [0,2\pi)$  defined as
\begin{equation}\label{theta}
\cos(\theta(\bx))=\frac{x}{|\bx|},
 \qquad \sin(\theta(\bx))=\frac{y}{|\bx|}, \qquad
0\ne \bx=(x,y)\in {\mathbb R}^2,
\end{equation}
we have the following lemma for the equivalence of the RDL (\ref{eqn: reduced1})
\cite{BT, BT1}:
\begin{lem}
\label{lem: equi_dir_rdl2}.
For any fixed $X\in \mathcal{D}^M$, we have the following identity
\begin{equation}\label{equi_dir}
J\nabla_\bx Q\left(\bx;X\right)=\nabla_\bx
R\left(\bx;X\right)=J\nabla_\bx H\left(\bx;X\right), \qquad \bx\in
\mathcal{D},
\end{equation}
which immediately implies the equivalence between system  (\ref{eqn: reduced1}) and the following two systems:
for $t>0$
\begin{align*}
&
 (\alpha I+\beta n_{j} J) \frac{d}{dt}\bx_{j}(t)=2n_j\left[J\nabla_\bx H\left(\bx;X\right)|_{\bx=
\bx_{j}(t)}+\sum_{l=1\& l\ne  j}^Mn_{l}\frac{\bx_{j}(t)-\bx_{l}(t)}{|\bx_{j}(t) -\bx_{l}(t)|^2}\right], \\
&
  (\alpha I+\beta n_{j} J) \frac{d}{dt}\bx_{j}(t)=2n_j\left[J\nabla_\bx Q\left(\bx;X\right)|_{\bx=
\bx_{j}(t)}+\sum_{l=1\& l\ne  j}^Mn_{l}\frac{\bx_{j}(t)-\bx_{l}(t)}{|\bx_{j}(t) -\bx_{l}(t)|^2}\right].
\end{align*}
\end{lem}
\noindent {\bf Proof.}
For any fixed $X\in \mathcal{D}^M$, since $Q$ is a harmonic function, there
exists a function $\varphi_1(\bx)$ such that
$$J\nabla_\bx Q\left(\bx;X\right)=\nabla\varphi_1(\bx),\qquad \bx\in \mathcal{D}.$$
Thus, $\varphi_1(\bx)$ satisfies the Laplace equation
\begin{equation}\label{proof_dir1}
    \Delta\varphi_1(\bx)=\nabla\cdot(J\nabla_\bx
    Q(\bx;X))=\partial_{yx}\varphi_1(\bx)-\partial_{xy}\varphi_1(\bx)=0, \quad \bx\in \mathcal{D},
\end{equation}
with the following Neumann BC
\begin{equation}\label{proof_dir2}
   \partial_{\nu}\varphi_1(\bx)=(J\nabla_\bx Q(\bx;X))\cdot\nu=
   \nabla_\bx
   Q(\bx;X)\cdot\nu_\perp=\partial_{\nu_\perp}Q(\bx;X),
   \quad\bx\in\partial\mathcal{D}.
\end{equation}
Noticing (\ref{harmRC_dir}), we obtain for $\bx\in\partial\mathcal{D}$,
\begin{equation}\label{proof_dir3}
   \partial_{\nu}\varphi_1(\bx)=
   \partial_{\nu_\perp} \omega(\bx) -\frac{\partial}{\partial
  \nu_\perp}\sum_{l=1}^{M}n_{l}\theta(\bx-\bx_{l})=
  \partial_{\nu_\perp} \omega(\bx) -\frac{\partial}{\partial
  \nu}\sum_{l=1}^{M}n_{l}\ln|\bx-\bx_{l}|.
\end{equation}
Combining (\ref{proof_dir1}), (\ref{proof_dir3}),
(\ref{harmR}) and (\ref{harmRC}), we get
\begin{equation}\label{proof_dir4}
\Delta(R(\bx;X)-\varphi_1(\bx))=0, \quad \bx\in\mathcal{D}, \qquad
\partial_{\nu}\left(R(\bx;X)-\varphi_1(\bx)\right)=0,\quad \bx\in\partial\mathcal{D}.
\end{equation}
Thus
$$R(\bx;X)=\varphi_1(\bx)+{\rm constant}, \qquad \bx\in\mathcal{D},$$
which immediately implies the first equality in (\ref{equi_dir}).

Similarly, since $H$ is a harmonic function, there exists a function
$\varphi_2(\bx)$ such that
$$J\nabla_\bx H\left(\bx;X\right)=\nabla\varphi_2(\bx),\qquad \bx\in \mathcal{D}.$$
Thus, $\varphi_2(\bx)$ satisfies the Laplace equation
\begin{equation}\label{proof_dir5}
    \Delta\varphi_2(\bx)=\nabla\cdot(J\nabla_\bx
    H(\bx;X))=\partial_{yx}\varphi_2(\bx)-\partial_{xy}\varphi_2(\bx)=0, \qquad \bx\in \mathcal{D},
\end{equation}
with the following Neumann BC
\begin{equation}\label{proof_dir6}
   \partial_{\nu}\varphi_2(\bx)=(J\nabla_\bx H(\bx;X))\cdot\nu=
   \nabla_\bx
   H(\bx;X)\cdot\nu_\perp=\partial_{\nu_\perp}H(\bx;X),
   \qquad\bx\in\partial\mathcal{D}.
\end{equation}
Combining (\ref{proof_dir5}), (\ref{proof_dir6}), (\ref{harmR}),
(\ref{harmRC}) and (\ref{harmRC_tan}), we get
\begin{equation}\label{proof_dir7}
\Delta(R(\bx;X)-\varphi_2(\bx))=0, \quad \bx\in\mathcal{D}, \qquad
\partial_{\nu}\left(R(\bx;X)-\varphi_2(\bx)\right)=0,\quad \bx\in\partial\mathcal{D}.
\end{equation}
Thus
$$R(\bx;X)=\varphi_2(\bx)+{\rm constant}, \qquad \bx\in\mathcal{D},$$
which immediately implies the second equality in (\ref{equi_dir}).
\hfill $\square$

\subsection{Under homogeneous Neumann boundary condition}\label{sec: rdl_neu}

For the CGLE (\ref{cgle}) with initial condition (\ref{ini_con}) under homogeneous Neumann
BC (\ref{neu-bc}), it has been derived formally and rigorously \cite{JS,KMMS,CJ} that
$W_{\rm bc}(X)$ in the renormalized energy $(\ref{re_ener})$ admit the form:
\begin{equation}\label{renorm_ener_neu}
W_{\rm bc}(X)=W_{\rm nbc}(X):= -\sum_{j=1}^{M}n_{j}\tilde{R}(\bx_{j};X),
\end{equation}
  and by using the following identity
\begin{equation}
\nabla_{\bx_j}\left[W(X)+ W_{\rm
nbc}(X)\right]=-2n_{j}\nabla_{\bx}\left[\tilde{R}(\bx;X)+\sum_{l=1\&
l\ne j}^Mn_{l}\ln|\bx-\bx_{l}|\right]_{\bx_{j}},
\end{equation}
we have the following simplified equivalent form for (\ref{eqn: rdl_gen}):
\begin{lem}
For $1\le j\le M$ and $t>0$, system (\ref{eqn: rdl_gen}) can be simplified as
\begin{equation}\label{eqn: reduced-neu1}
(\alpha I+\beta n_{j} J) \frac{d}{dt}\bx_{j}(t)=2n_j\left[\nabla_\bx \tilde{R}\left(\bx;X\right)|_{\bx
=\bx_{j}(t)}+\sum_{l=1\& l\ne j}^Mn_{l}\frac{\bx_{j}(t)-\bx_{l}(t)}{|\bx_{j}(t) -\bx_{l}(t)|^2}\right].
\end{equation}
\end{lem}
\noindent Moreover,  for any fixed $X\in \mathcal{D}^M$, by introducing function $\tilde{H}(\bx,X)$
and  $\tilde{Q}(\bx,X)$  that both are harmonic in $\bx$ satisfying respectively
the boundary condition \cite{SJYM1, SJYM2, SJYM3, LX}:
\begin{align}
&\label{harmRC_neu_tan}
\frac{\partial \tilde{H}(\bx;X)}{\partial \nu_\perp} =
-\frac{\partial}{\partial
  \nu}\sum_{l=1}^{M}n_{l}\theta(\bx-\bx_{l}), \qquad
  \bx\in \partial\mathcal{D},\\
 &\label{harmRC_neu}
\frac{\partial \tilde{Q}(\bx;X)}{\partial \nu}
=-\frac{\partial}{\partial
  \nu}\sum_{l=1}^{M}n_{l}\theta(\bx-\bx_{l}), \qquad
  \bx\in \partial\mathcal{D},
\end{align}
with the function $\theta: \ {\mathbb R}^2 \to [0,2\pi)$  being defined in (\ref{theta}),
we have the following lemma for the equivalence of the RDL (\ref{eqn: reduced-neu1})
\cite{BT, BT1}:
\begin{lem}
For any fixed $X\in \mathcal{D}^M$, we have the following identity
\begin{equation}\label{equi_neu}
J\nabla_\bx \tilde{Q}\left(\bx;X\right)=\nabla_\bx
\tilde{R}\left(\bx;X\right)=J\nabla_\bx \tilde{H}\left(\bx;X\right), \qquad \bx\in
\mathcal{D},
\end{equation}
which immediately implies the equivalence of  system  (\ref{eqn: reduced-neu1}) and the following two systems:
for $t>0$
\begin{align*}
&
 (\alpha I+\beta n_{j} J) \frac{d}{dt}\bx_{j}(t)=2n_j\left[\nabla_\bx \tilde{H}\left(\bx;X\right)|_{\bx=
\bx_{j}(t)}+\sum_{l=1\& l\ne  j}^Mn_{l}\frac{\bx_{j}(t)-\bx_{l}(t)}{|\bx_{j}(t) -\bx_{l}(t)|^2}\right], \\
&
  (\alpha I+\beta n_{j} J) \frac{d}{dt}\bx_{j}(t)=2n_j\left[J\nabla_\bx \tilde{Q}\left(\bx;X\right)|_{\bx=
\bx_{j}(t)}+\sum_{l=1\& l\ne  j}^Mn_{l}\frac{\bx_{j}(t)-\bx_{l}(t)}{|\bx_{j}(t) -\bx_{l}(t)|^2}\right].
\end{align*}
\end{lem}
\begin{proof}
Follow the line in the proof of lemma \ref{lem: equi_dir_rdl1} and we omit the details here for brevity.
\end{proof}

\section{Numerical methods}
\label{sec: num_method}

In this section, we will give a brief outline for discussing by some efficient and accurate
numerical methods how to solve the
CGLE (\ref{cgle}) on either a rectangle or a disk
with initial condition (\ref{ini_con}) and under either Dirichlet BC (\ref{dir})
or homogeneous Neumann BC (\ref{neu-bc}).
The key idea in our numerical methods are based on: (i) applying a time-splitting
technique which has been widely used for nonlinear partial
differential equations to decouple the nonlinearity in
the CGLE \cite{GRTP, S, BJM, YZWBQD1}; and (ii) adopting proper finite
difference/element and/or spectral method to discretize a gradient
flow with constant coefficients \cite{BDZ, BT, BT1}.

Let $\tau >0$
be the time step size, denote $t_{n}=n\tau$ for $n\ge0$. For
$n=0,1,\ldots$, from time $t=t_{n}$ to $t=t_{n+1}$, the CGLE (\ref{cgle}) is
solved in two splitting steps. One first solves
\begin{equation}\label{step1}
(\lambda_\vep+i\beta)\partial_{t}\psi^{\vep}(\bx,t)
=\frac{1}{\vep^{2}}(1-|\psi^{\vep}|^{2})\psi^{\vep},
\qquad \bx\in \mathcal{D}, \quad t\ge t_n,
\end{equation}
for the time step of length $\tau$, followed by
solving
\begin{equation}\label{step2}
(\lambda_\vep+i\beta)\partial_{t}\psi^{\vep}(\bx,t)= \Delta \psi^{\vep},
\qquad \bx\in\mathcal{D}, \quad t\ge t_n,
\end{equation}
for the same time step. Methods to discretize  equation (\ref{step2})  will be outlined later. For
$t\in[t_n,t_{n+1}]$, we can easily obtain from equation (\ref{step1})
the following ODE for $\rho^{\vep}(\bx,t)=|\psi^{\vep}(\bx,t)|^2$:
\begin{equation}\label{eqn: rho_step1}
\partial_t\rho^{\vep}(\bx,t)=\eta[1-\rho^{\vep}(\bx,t)]\rho^{\vep}(\bx,t),\quad \bx\in\mathcal{D}, \quad t_n\le t\le t_{n+1},
\end{equation}
where $\eta=2\lambda_\vep/\vep^2(\lambda_\vep^2+\beta^2)$.
Solving equation (\ref{eqn: rho_step1}), we have
\begin{equation}\label{slrho2}
\rho^{\vep}(\bx,t)=\frac{\rho^{\vep}(\bx,t_n)}{\rho^{\vep}(\bx,t_n)+(1-\rho^{\vep}(\bx,t_n))\exp[-\eta (t-t_n)]}.
\end{equation}
Plugging (\ref{slrho2}) into (\ref{step1}), we can integrate it exactly to get
\begin{equation}\label{explicit}
\psi^{\vep}(\bx,t)=\psi^{\vep}(\bx,t_{n})
\sqrt{\hat{P}(\bx,t)}\exp\left[-\frac{i\beta}{2\lambda_\vep^2}\ln \hat{P}(\bx,t) \right],
\end{equation}
where
\begin{equation}
\label{p}
\hat{P}(\bx,t)=
\frac{1}{|\psi^{\vep}(\bx,t_n)|^2+(1-|\psi^{\vep}(\bx,t_n)|^2)\exp(-\eta(t-t_n))}.
\end{equation}

We remark here that, in practice, we always use the second-order Strang
 splitting \cite{S}, that is, from time $t=t_n$ to $t=t_{n+1}$: (i)
 evolve (\ref{step1}) for half time step $\tau/2$ with initial data
 given at $t=t_n$; (ii) evolve (\ref{step2}) for one step
 $\tau$ starting with the new data;
 and (iii) evolve (\ref{step1}) for half time step $\tau/2$ again with
 the newer data.

When  $\Omega=[a,b]\times[c,d]$ is a rectangular domain, we denote
$h_{x}$=$\frac{b-a}{N}$ and $h_{y}$=$\frac{d-c}{L}$ with
$N$ and $L$ being two even positive integers as the mesh sizes in $x-$direction and
$y-$direction, respectively.  Similar to the
discretization of the gradient flow with constant coefficient \cite{BT},
when the Dirichlet BC (\ref{dir}) is used for the equation (\ref{step2}),
it can be discretized  by using the 4th-order compact
finite difference discretization for spatial derivatives followed by
a Crank-Nicolson (CNFD) scheme for temporal derivative \cite{BT, BT1};
and when homogeneous Neumann BC (\ref{neu-bc}) is used for the equation (\ref{step2}),
it can be discretized  by
using cosine spectral discretization for spatial derivatives followed by
integrating in time {\sl exactly} \cite{BT, BT1}. The details are omitted here
for brevity. Combining the
CNFD and cosine psedudospectral
discretization for Dirichlet and homogeneous Neumann BC, respectively,
with the second order Strang splitting,
we can obtain time-splitting Crank-Nicolson finite difference
(TSCNFD) and time-splitting cosine psedudospectral
(TSCP) discretizations
for the CGLE (\ref{cgle}) on a rectangle with
Dirichlet BC (\ref{dir}) and homogeneous Neumann BC (\ref{neu-bc}), respectively.
Both TSCNFD and TSCP discretizations
are unconditionally stable, second order in time,
the memory cost is $O(NL)$ and the computational cost per time step
is $O\left(NL\ln(NL)\right)$. In addition, TSCNFD is fourth order in space
and TSCP is spectral order in space.

When  $\Omega=\{\bx \ |\ |\bx|<R\}:=B_R({\bf 0})$ is a disk with $R>0$ a fixed
constant. Similar to the
discretization of the GPE with an angular momentum rotation
\cite{B,BDZ,YZWBQD1} and/or the gradient flow
with constant coefficient \cite{BT}, it is natural to adopt the polar coordinate
$(r,\theta)$ in the numerical discretization by using
the standard Fourier pseduospectral method in $\theta$-direction
\cite{JSTT}, finite element method in $r$-direction, and
Crank-Nicolson method in time \cite{B,BDZ,YZWBQD1}.
Again, the details are omitted here
for brevity.

\section{Numerical results under Dirichlet BC}
\label{sec: CGLE_Dir}

In the section, we report numerical results
for vortex interactions of the CGLE (\ref{cgle}) under
the Dirichlet BC (\ref{dir}) and compare
them with those obtained from the corresponding RDLs.
For simplicity, from now on,
we assume that the parameters $\alpha=1$ in (\ref{lam_vep}) and $\beta=1$ in (\ref{cgle}).

We study how the
dimensionless parameter $\vep$, initial setup, boundary value and geometry
of the domain $\mathcal{D}$ affect the dynamics and interaction of vortices.
For a given bounded domain $\mathcal{D}$,
the CGLE (\ref{cgle}) is unchanged by the re-scaling $\bx\to l\bx$,
 $t\to l^2t$ and $\vep\to l\vep$ with $l$ the diameter of $\mathcal{D}$.
Thus without lose of generality, hereafter, without specification, we always
assume that the diameter of $\mathcal{D}$ is $O(1)$.
The function $g$ in the Dirichlet BC (\ref{dir}) is given as
\begin{equation*}
g(\bx)=e^{i(h(\bx)+\sum_{j=1}^{M}n_j\theta(\bx-\bx^0_j))},\qquad \bx\in\partial\mathcal{D},
\end{equation*}
and
the initial condition $\psi_0^\vep$ in (\ref{ini_con}) is
chosen as
\begin{equation} \label{initdbc0}
\psi_0^\vep(\bx)=\psi_0^\vep(x,y)=e^{ih(\bx)}\prod_{j=1}^{M}
\phi^\vep_{n_j} (\bx-\bx^0_j), \qquad \bx=(x,y)\in \overline{\mathcal{D}},
\end{equation}
where $M>0$ is the total number of vortices in the initial data, the phase shift
$h(\bx)$ is a harmonic function, $\theta(\bx)$ is defined in
(\ref{theta}) and for $j=1,2,\ldots,M$, $n_j=1$
or $-1$, and $\bx^0_j=(x_j^0,y_j^0)\in \mathcal{D}$ are the winding
number and initial location of the $j$-th vortex, respectively.
Moreover, for $j=1,\ldots,M$, the function $\phi^\vep_{n_j}(\bx)$ is chosen as
a single vortex centered at the origin with winding number $n_j=1$ or $-1$
which was computed numerically and depicted in section 4 in \cite{BT, BT1}.
In addition, in the following sections, we mainly consider six different modes of
the phase shift $h(\bx):$
\begin{itemize}
\item Mode 0: $h(\bx)=0,$   \qquad \qquad \qquad \quad\;\; Mode 1: $h(\bx)=x+y, $
\item Mode 2: $h(\bx)=x-y, $ \qquad \qquad \qquad  Mode 3: $h(\bx)=x^2-y^2, $
\item Mode 4: $h(\bx)=x^2-y^2+2xy,$  \qquad\   Mode 5: $h(\bx)=x^2-y^2-2xy.$
\end{itemize}

To simplify our discussion, for $j=1,2,\ldots,M$, hereafter we let
$\bx^{\rm r}_j(t)$ be the $j$-th vortex center in the
reduced dynamics and denote $d_j^\vep(t)=|\bx^\vep_j(t)-\bx^{\rm r}_j(t)|$
as the difference of the vortex centers in the CGLE dynamics and reduced dynamics.
Furthermore, in the presentation of figures, the initial location of
a vortex with winding number $+1$, $-1$ and the location that two vortices merge are marked
as `+', `$\circ$' and `$\diamond$', respectively. Finally, in our computations, if not specified,
we take $\mathcal{D}=[-1,1]^2$, mesh sizes
$h_x=h_y=\frac{\vep}{10}$ and time step $\tau=10^{-6}$. The CGLE
(\ref{cgle}) with (\ref{dir}) and (\ref{ini_con})
is solved by the method TSCNFD presented in section \ref{sec: num_method}.

\subsection{Single vortex}
\label{sec: cgle_NRfSVD}

In this subsection,  we present numerical results of the motion of a single quantized
vortex in the CGLE  dynamics and the corresponding reduced dynamics. We choose the parameters as
$M=1$, $n_1=1$ in (\ref{initdbc0}). To study how the initial phase shift $h(\bx)$,
initial location of the vortex $\bx_0$ and domain geometry affect the motion of the vortex and to
understand the validity of the RDL, we consider the following 16 cases:
\begin{itemize}
\item Case I-III: $\bx_1^0=(0,0)$, $h(\bx)$ is chosen as Mode 1, 2 or 3, and $\mathcal{D}$ is type I;
\item Case IV-VIII: $\bx_1^0=(0.1,0)$, $h(\bx)$ is chosen as Mode 1, 2, 3, 4 or 5, and $\mathcal{D}$ is type I;
\item Case IX-XII: $\bx_1^0=(0.1,0.2)$, $h(\bx)$ is chosen as Mode 2, 3, 4 or 5, and $\mathcal{D}$ is type I;
\item Case XIII-XIV:  $\bx_1^0=(0,0)$, $h(\bx)=x+y$ and $\mathcal{D}$ is chosen as type II or III;
\item Case XV-XVI: $\bx_1^0=(0.1,0.2)$, $h(\bx)=x^2-y^2$ and $\mathcal{D}$ is chosen as type II or III,
\end{itemize}
where three different types of domains $\mathcal{D}$ are considered:
type I: $\mathcal{D}=[-1,1]\times[-1,1],$
type II: $\mathcal{D}=[-1,1]\times[-0.65,0.65],$
type III: $\mathcal{D}=B_1(0).$

\begin{figure}[t!]
 \centerline{(a)\psfig{figure=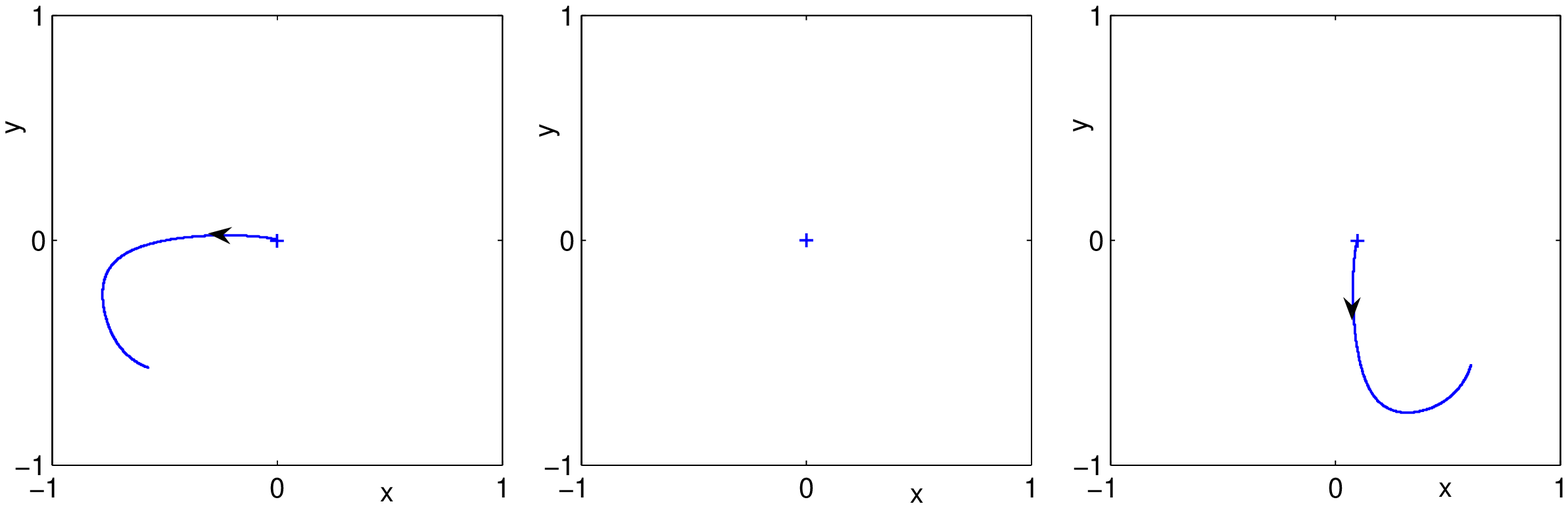,height=4cm,width=12.5cm,angle=0}}
 \centerline{(b)\psfig{figure=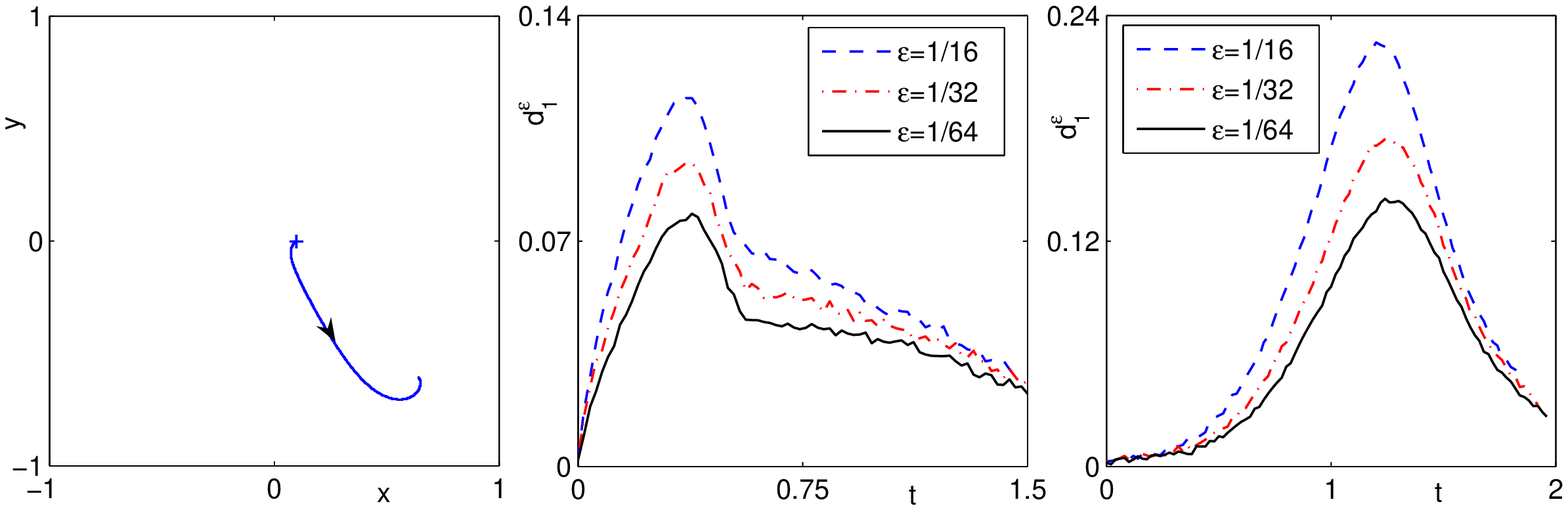,height=4cm,width=12.5cm,angle=0}}
 \caption{ Trajectory of the vortex center in CGLE under Dirichlet BC when $\vep=\frac{1}{32}$
 under four cases, i.e. cases II-IV and VI, and the time evolution of $d_1^\vep$  for different $\vep$ under cases
 II and VI (from left to right and then from top to bottom) in section \ref{sec: cgle_NRfSVD}.}
\label{fig: cgle-one-vortex-conver-D}
\end{figure}

Fig. \ref{fig: cgle-one-vortex-conver-D}
shows the trajectory of the vortex center when $\vep=\frac{1}{32}$
for cases II-IV and VI as well as the time evolution of $d_1^\vep(t)$  for different $\vep$ for cases II and VI,
and the trajectory of the vortex center under cases V-VIII and cases IX-XII are, respectively, shown
by Fig. \ref{fig: cgle-one-vortex-BC_Eff-D} and Fig. \ref{fig: cgle-one-vortex-Geo_Eff-D} when
$\vep=\frac{1}{32}$ in CGLE. Based on these ample numerical results (although some results are not shown here for brevity),
we made the following observations for the single vortex dynamics:
\begin{figure}[t!]
 \centerline{\psfig{figure=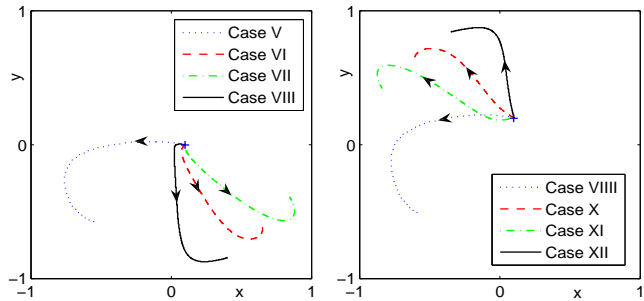,height=4.2cm,width=9cm,angle=0}}
  \caption{Trajectory of the vortex center
   in CGLE under Dirichlet BC when $\vep=\frac{1}{32}$  under cases V-VIII (left) and
   cases IX-XII (right) in section \ref{sec: cgle_NRfSVD}.}
\label{fig: cgle-one-vortex-BC_Eff-D}
 \end{figure}
(i). When $h(\bx)\equiv0$, the vortex center doesn't move, which is similar to the vortex dynamics in
the whole space in GLE and NLSE dynamics.
(ii). When $h(\bx)=(x+by)(x-\frac{y}{b})$ with $b\neq 0$, the vortex does not move if
$\bx_0=(0,0)$, while it does move if $\bx_0\neq (0,0)$ (please see case III and VI for $b=1$).
This is the same as the phenomena in GLE and NLSE dynamics.
(iii). When $h(\bx)\ne0$ and $h(\bx)\neq(x+by)(x-\frac{y}{b})$ with $b\neq 0$, in general,
the vortex center does move to a different point from its initial location and then it will stay there forever.
This is quite different from the corresponding case in the whole space, since in that case a single vortex may move to infinity
under the initial data (\ref{initdbc0}) with  $h(\bx)\ne0$.
(iv). In general, the initial location, the geometry of the domain and the boundary
value will take effect on the motion of the vortex center.
(v). When $\vep\rightarrow 0$, the dynamics of the vortex center
in the CGLE dynamics converges uniformly in time to
that in the reduced dynamics (see Fig. \ref{fig: cgle-one-vortex-conver-D})
which verifies numerically the validation
of the RDLs. In fact, based on our
extensive numerical experiments, the motion of
the vortex center from the RDLs agrees with those
from the CGLE dynamics qualitatively when $0<\vep<1$ and quantitatively
when $0<\vep\ll 1$.

\begin{figure}[t!]
 \centerline{(a)\psfig{figure=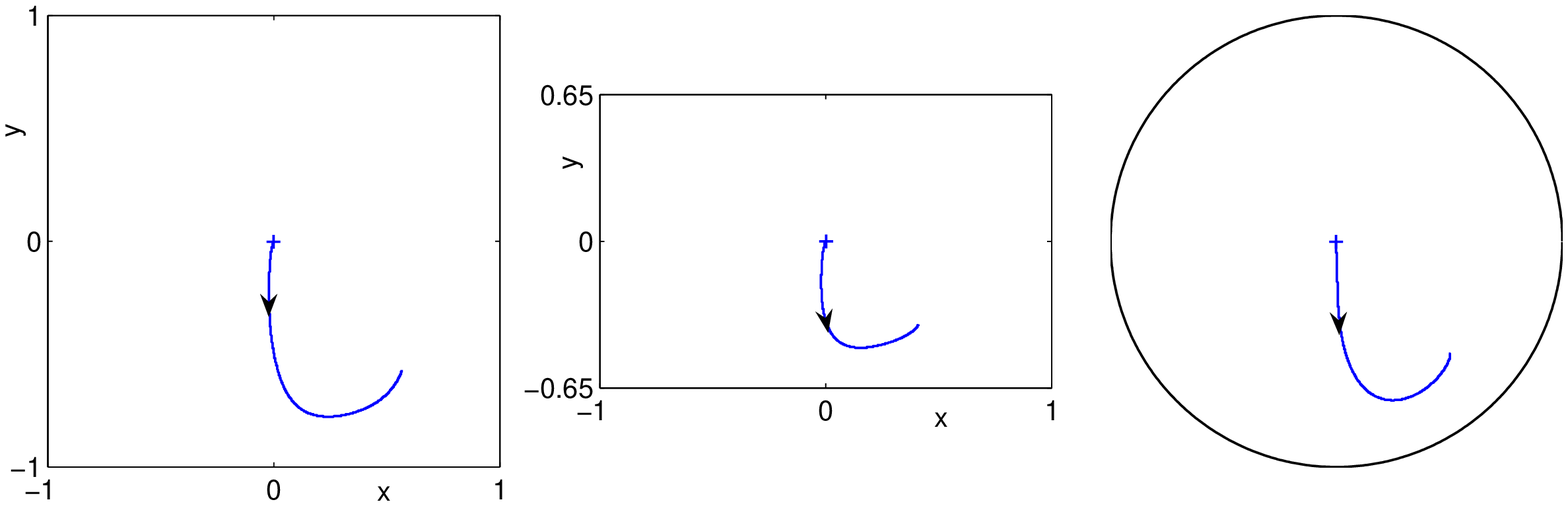,height=4cm,width=12.5cm,angle=0}}
 \centerline{(b)\psfig{figure=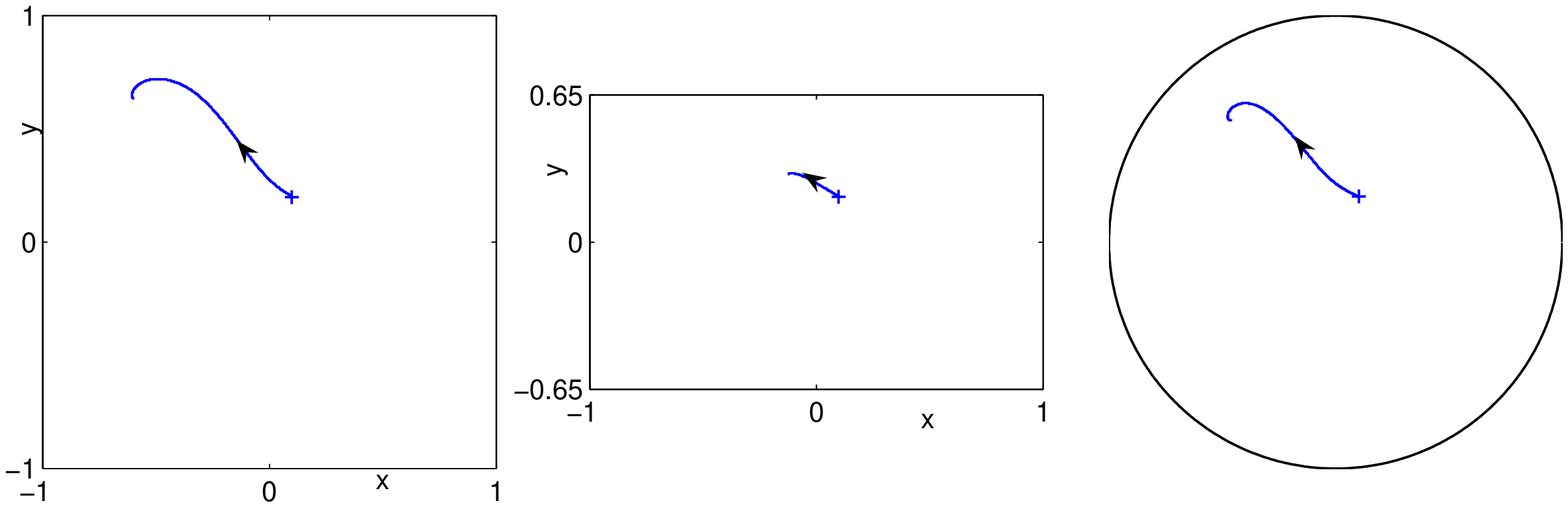,height=4cm,width=12.5cm,angle=0}}
 \caption{Trajectory of the vortex center in CGLE under Dirichlet BC when $\vep=\frac{1}{32}$
 for cases: (a) I, XIII, XIV,  (b) X, XV, XVI (from left to right) in section \ref{sec: cgle_NRfSVD}.}
\label{fig: cgle-one-vortex-Geo_Eff-D}
 \end{figure}

\subsection{Vortex pair}
\label{sec: cgle_NRfVPD}

Here we present numerical results of the interaction of vortex pair in
the CGLE  dynamics and its corresponding reduced dynamics. In the following numerical simulations of the subsection,
we take $M=2$, $n_1=n_2=1$, $\bx_1^0=(-0.3,0)$
and $\bx_2^0=(0.3,0)$ in (\ref{initdbc0}). Fig. \ref{fig: cgle-vpd-modes-ep25}
depicts the trajectory of the vortex centers and their corresponding time evolution of the GL functionals
when $\vep=\frac{1}{25}$ in the CGLE  with different $h(\bx)$ in (\ref{initdbc0}), and Fig.
\ref{fig: cgle-vpd-dens-trace-err} shows contour plots of $|\psi^{\vep}(\bx,t)|$ for
$\vep=\frac{1}{25}$ at different times as well as
the time evolution of $\bx_1^{\vep} (t)$, $\bx_1^{\rm r} (t)$
and $d_1^{\vep} (t)$ for different $\vep$ with $h(\bx)=0$ in (\ref{initdbc0}).

\begin{figure}
 \centerline{(a)\psfig{figure=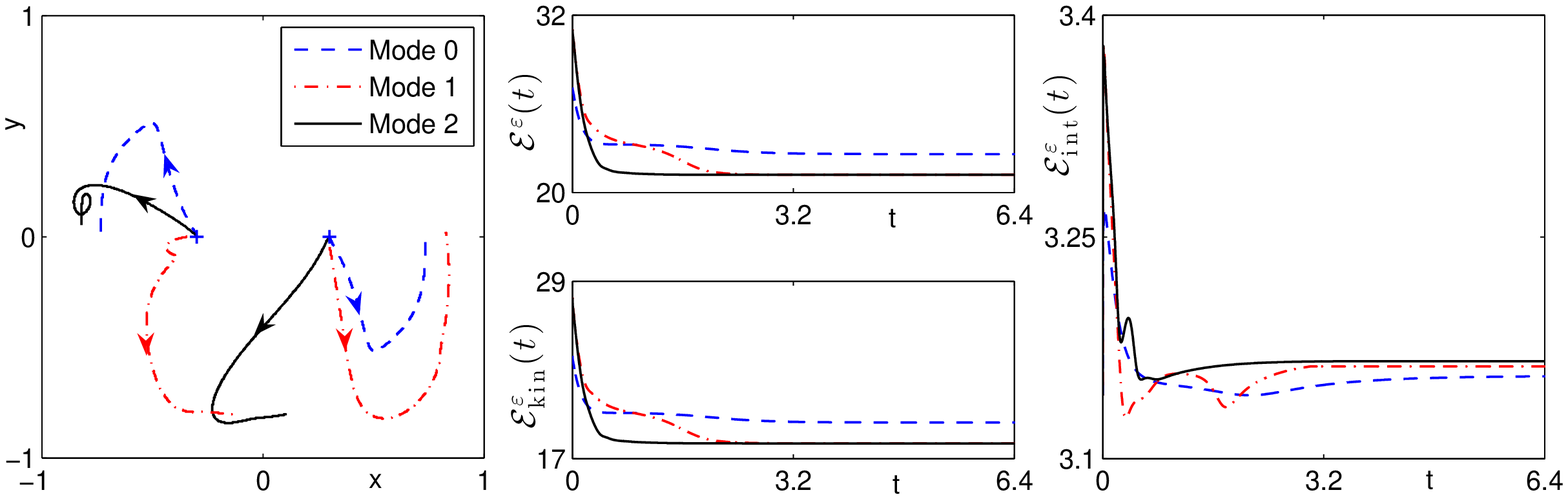,height=4cm,width=12.5cm,angle=0}}
 \centerline{(b)\psfig{figure=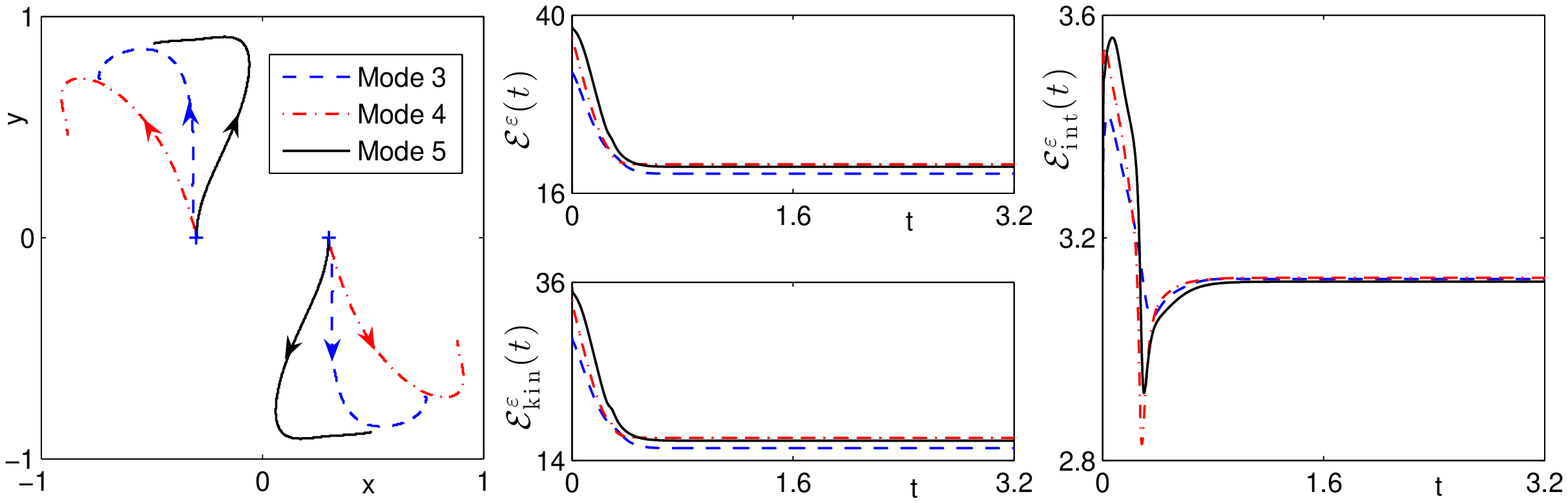,height=4cm,width=12.5cm,angle=0}}
 \caption{ Trajectory of the vortex centers (a) and their corresponding time evolution of the GL functionals (b) in  CGLE dynamics
  under Dirichlet BC when $\vep=\frac{1}{25}$  with different $h(\bx)$ in (\ref{initdbc0})  in section \ref{sec: cgle_NRfVPD}.}
\label{fig: cgle-vpd-modes-ep25}
 \end{figure}

According to our ample numerical experiments,
we made the following observations for the interaction
of vortex pair in the CGLE dynamics with Dirichlet BC:
(i). The motion of the vortex pair may be thought of as
a kind of combination between that in the GLE and NLSE dynamics with Dirichlet BC.
From Figs. \ref{fig: cgle-vpd-modes-ep25}-\ref{fig: cgle-vpd-dens-trace-err},
we observed that the two vortices undergo a repulsive interaction, and
they first rotate with each other and meanwhile move apart from each
other towards the boundary of the domain, then stop somewhere near the
boundary, which indicates that the boundary of the domain imposes a repulsive force
on the two vortices.
As shown in previous studies \cite{BT,BT1}, a vortex pair in the GLE dynamics moves outward along the line that connects with the two vortices and finally stay static near the boundary,
while in the NLSE dynamics the two vortices always rotate around each other periodically. In fact, based on our extensive numerical
results, we found that the larger the value $\beta$ (or $\alpha$) is, the closer the motion in CGLE dynamics is
to that in NLSE (or GLE) dynamics,  which gives the sufficient numerical evidence for our above conclusion.
(ii).  The phase shift $h(\bx)$ affects the motion of the vortices significantly.
When $h(\bx)=(x+by)(x-\frac{y}{b})$ with $b\neq 0$, the vortices will move outward
symmetric with respect to the origin, i.e., $\bx_1(t)=-\bx_2(t)$ (see Fig. \ref{fig: cgle-vpd-modes-ep25}).
(iii). When $\vep\rightarrow 0$, the dynamics of the two vortex centers
in the CGLE dynamics converges uniformly in time to
that in the reduced dynamics (see Fig. \ref{fig: cgle-vpd-dens-trace-err})
 which verifies numerically the validation
of the RDLs in this case. In fact, based on our
extensive numerical experiments, the motions of
the two vortex centers from the RDLs agree with those
from the CGLE dynamics qualitatively when $0<\vep<1$ and quantitatively
when $0<\vep\ll 1$.
(iv). During the dynamics evolution of CGLE, the GL functional and its kinetic part
decrease as the time evolves, its interaction part changes dramatically
when $t$ is small, and when $t\to \infty$,  all the three quantities
converge to constants (see Fig. \ref{fig: cgle-vpd-modes-ep25}),
which immediately indicates that a steady state solution will be reached when $t\to\infty$.

\begin{figure}
\centerline{\quad\psfig{figure=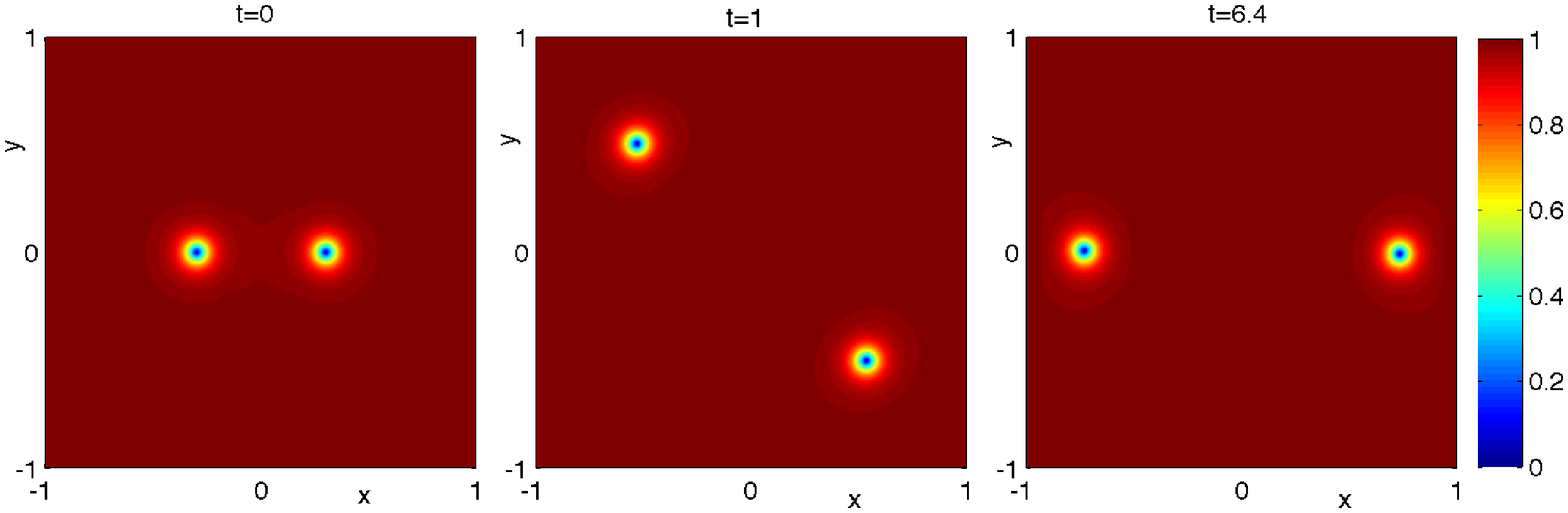,height=4cm,width=12.5cm,angle=0}}
 \centerline{\psfig{figure=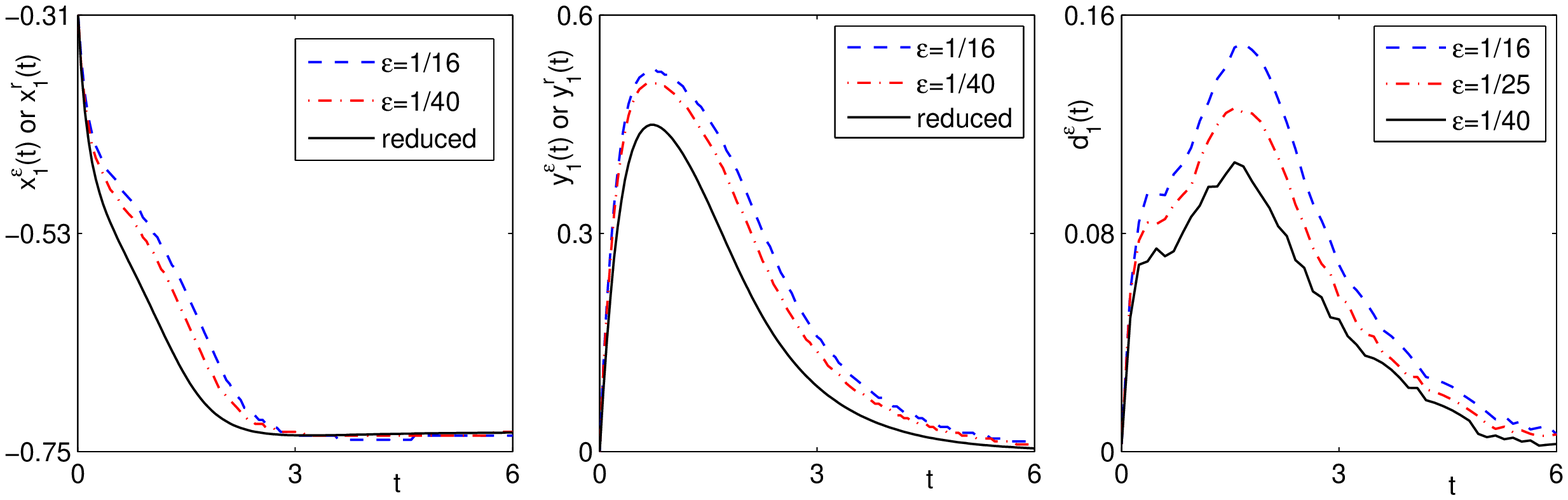,height=4cm,width=12.5cm,angle=0}}
 \caption{ Contour plot of $|\psi^{\vep}(\bx,t)|$ for $\vep=\frac{1}{25}$ at different times as well as  time evolution of $\bx_1^{\vep} (t)$ in  CGLE dynamics and $\bx_1^{\rm r} (t)$ in the reduced dynamics
  under Dirichlet BC with $h(\bx)=0$ in (\ref{initdbc0}) and their difference $d_1^{\vep} (t)$
for different $\vep$  in section \ref{sec: cgle_NRfVPD}.}
\label{fig: cgle-vpd-dens-trace-err}
 \end{figure}

\subsection{Vortex dipole}
 \label{sec: cgle_NRfVDD}

Here we present numerical results of the interaction of vortex dipole under
the CGLE  dynamics and its corresponding reduced
dynamical laws. We choose the parameters in the simulations as $M=2$, $n_1=-n_2=-1$, $\bx_2^0=-\bx_1^0=(0.3,0)$
in (\ref{initdbc0}). Fig. \ref{fig: cgle-vdd-modes-ep25}
depicts the trajectory of the vortex centers and their corresponding time evolution of the GL functionals
when $\vep=\frac{1}{25}$ in the CGLE with different $h(\bx)$ in (\ref{initdbc0}), and Fig.
\ref{fig: cgle-vdd-trace-err} shows
contour plot of $|\psi^{\vep}(\bx,t)|$ for $\vep=\frac{1}{25}$ at different times as well as
 the time evolution of $\bx_1^{\vep} (t)$, $\bx_1^{\rm r} (t)$
and $d_1^{\vep} (t)$ for different $\vep$ with $h(\bx)=0$ in (\ref{initdbc0}).
\begin{figure}[t!]
 \centerline{(a)\psfig{figure=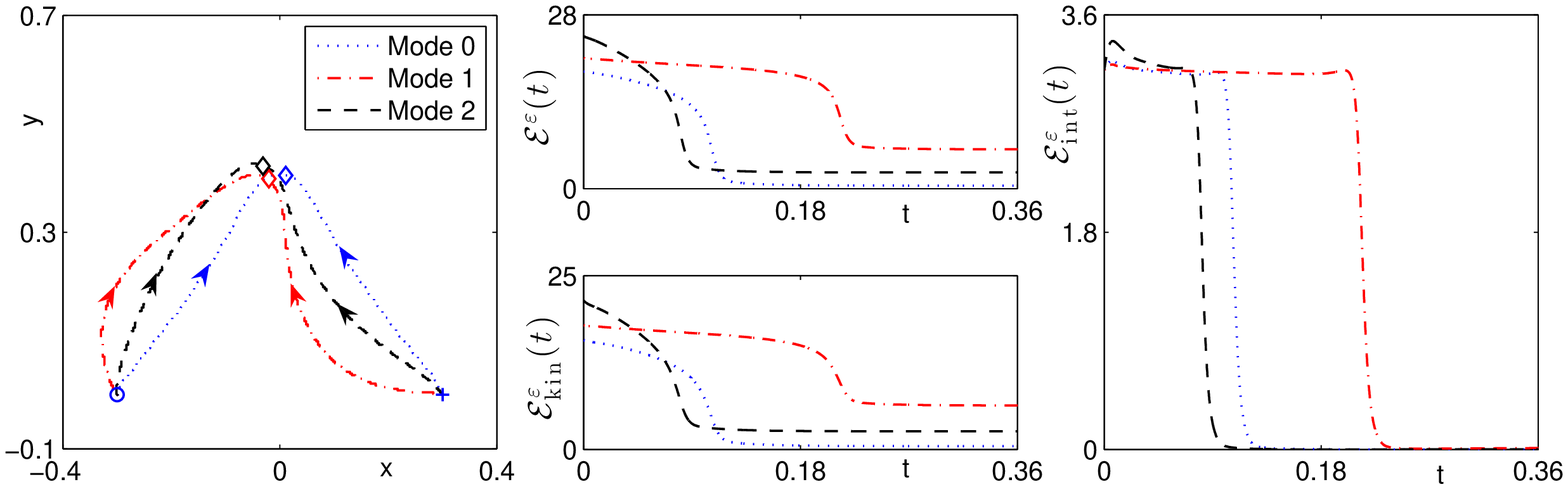,height=4cm,width=12.5cm,angle=0}}
 \centerline{(b)\psfig{figure=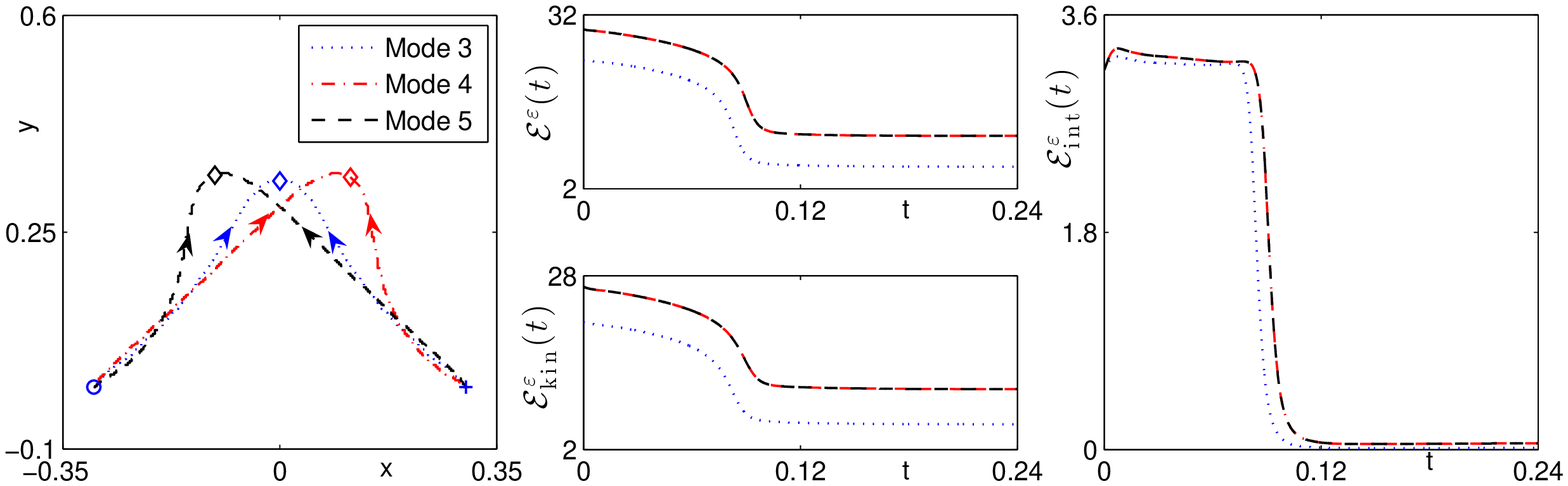,height=4cm,width=12.5cm,angle=0}}
 \caption{ Trajectory of the vortex centers (a) and their corresponding time evolution of the GL functionals (b) in  CGLE dynamics
  under Dirichlet BC when $\vep=\frac{1}{25}$  with different $h(\bx)$ in (\ref{initdbc0}) in section \ref{sec: cgle_NRfVDD}.}
\label{fig: cgle-vdd-modes-ep25}
\end{figure}

From Figs. \ref{fig: cgle-vdd-modes-ep25}-\ref{fig: cgle-vdd-trace-err} and
ample numerical experiments (not shown here for brevity),
we made the following observations for the interaction
of vortex dipole in the CGLE dynamics with Dirichlet BC:
(i). The two vortices undergo an
attractive interaction, they will collide and annihilate with each other.
(ii). The phase shift $h(\bx)$ and the initial distance of the two vortices
affect the motion of the vortices significantly. If $h(\bx)=0$,
regardless of where the vortices are initially located,
the vortex dipole will finally merge. However, similar as the
case in GLE dynamics, if $h(\bx)\not\equiv0$, say $h(\bx)=x+y$ for example,
there would be a critical distance $d_c^\vep$, which depends on the value of $\vep$,
that divides the motion of the vortex dipole into two cases: (a) if the initial distance between
the vortex dipole $|\bx_2^0-\bx_1^0|>d_c^\vep$, the vortex will never merge, they
will finally stay static and separate at somewhere near the boundary. (b) otherwise,
they do finally merge and  annihilate.
(iii). For $h(\bx)=0$,  when $\vep\rightarrow 0$, the dynamics of the two vortex centers
in the CGLE dynamics converges uniformly in time to
that in the reduced dynamics (see Fig. \ref{fig: cgle-vdd-trace-err}),
which verifies numerically the validation
of the RDLs in this case. In fact, based on our
extensive numerical experiments, the motions of
the two vortex centers from the RDLs agree with those
from the CGLE dynamics qualitatively when $0<\vep<1$ and quantitatively
when $0<\vep\ll 1$ before they merge.
(iv). During the dynamics evolution of CGLE, the GL functional decreases as
the time evolves, its kinetic and interaction parts don't change dramatically
when $t$ is small, while all the three quantities
converge to constants when $t\to \infty$.
Moreover, if finite time merging/annihilation happens, the GL functional and its
kinetic and interaction parts change significantly during the collision.
In addition, when $t\to\infty$, the interaction energy goes to $0$
which immediately implies that a steady state will be reached
in the form of $\phi^{\vep}(\bx)=e^{ic(\bx)}$,  where $c(\bx)$ is
a harmonic function satisfying
$c(\bx)|_{\partial \mathcal{D}}=h(\bx)+\sum_{j=1}^{M}n_j\theta(\bx-\bx^0_j)$.

\begin{figure}[t!]
\centerline{\quad\psfig{figure=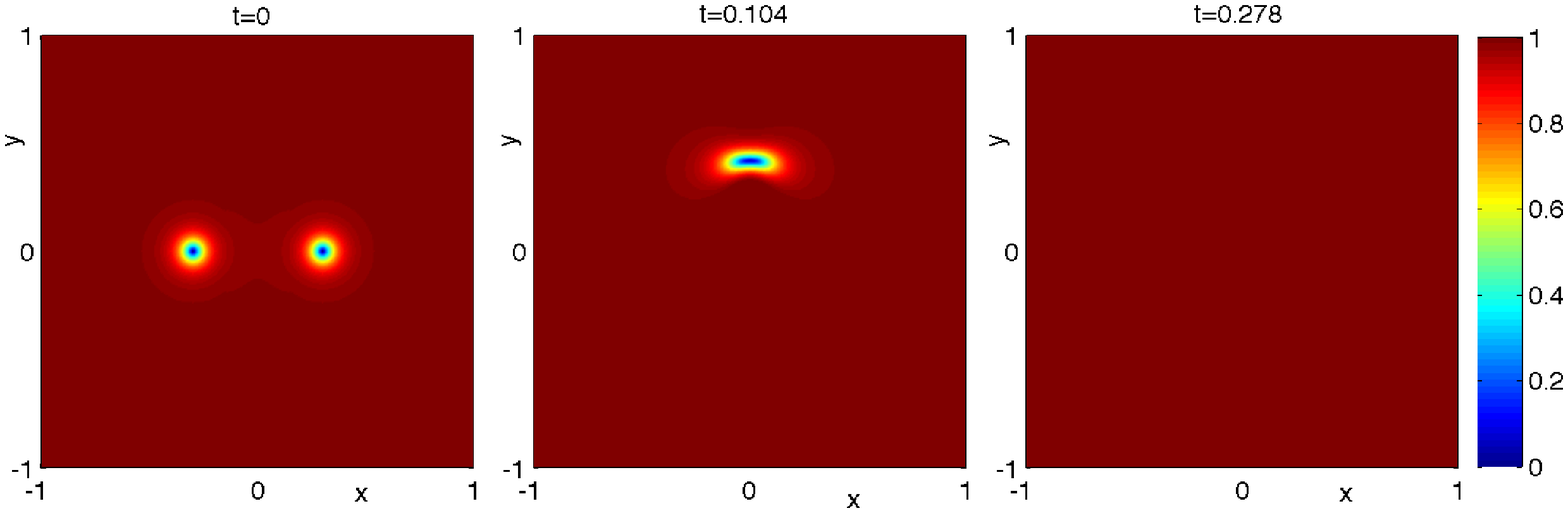,height=4cm,width=12.5cm,angle=0}}
 \centerline{\psfig{figure=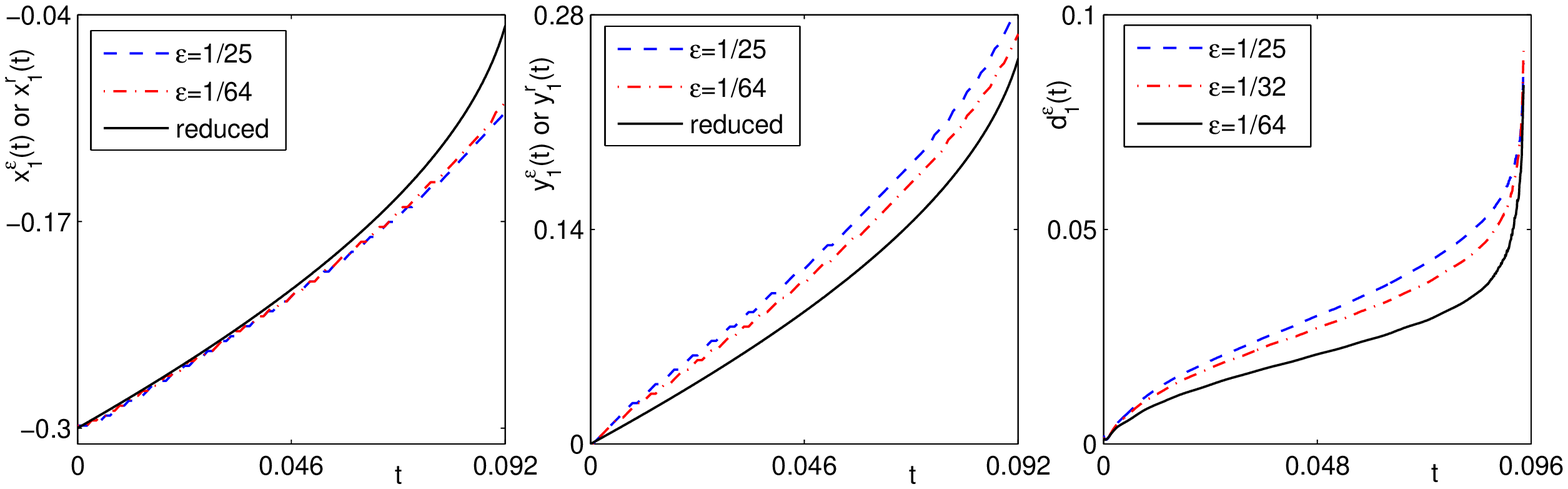,height=4cm,width=12.5cm,angle=0}}
 \caption{Contour plots of $|\psi^{\vep}(\bx,t)|$ for $\vep=\frac{1}{25}$ at different times as well as
 time evolution of $\bx_1^{\vep} (t)$ in  CGLE dynamics, $\bx_1^{\rm r} (t)$ in the reduced dynamics
  under Dirichlet BC with $h(\bx)=0$ in (\ref{initdbc0}) and their difference $d_1^{\vep} (t)$
for different $\vep$  in section \ref{sec: cgle_NRfVDD}.}
\label{fig: cgle-vdd-trace-err}
 \end{figure}

\subsection{Vortex lattice}
 \label{sec: cgle_NRfVLD}

Here we present numerical results about the interaction of vortex lattices under
the CGLE dynamics. We consider the following 15 cases:
case I. $M=3$, $n_1=n_2=n_3=1$,
		 $\bx_1^0=(0.5, 0)$;
         $\bx_2^0=(-0.25,\frac{\sqrt{3}}{4})$,
         $\bx_3^0=(-0.25,-\frac{\sqrt{3}}{4})$,
case II. $M=3$,  $n_1=n_2=n_3=1$, $\bx_1^0=(-0.4,0)$,
		 $\bx_2^0=(0,0)$, $\bx_3^0=(0.4,0)$;
case III. $M=3$, $n_1=n_2=n_3=1$, $\bx_1^0=(0,0.3)$,
		  $\bx_2^0=(0.15,0.15)$, $\bx_3^0=(0.3,0)$;		
case IV. $M=3$, $-n_1=n_2=n_3=1$,
         $\bx_1^0=(0.5, 0)$;
		 $\bx_2^0=(-0.25,\frac{\sqrt{3}}{4})$,
         $\bx_3^0=(-0.25,-\frac{\sqrt{3}}{4})$,
case V. $M=3$, $n_2=-1$, $n_1=n_3=1$,
		$\bx_1^0=(-0.4,0)$,
		$\bx_2^0=(0,0)$,
		$\bx_3^0=(0.4,0)$;
case VI. $M=3$, $n_1=-1$, $n_2=n_3=1$,
		 $\bx_1^0=(0.2,0.3)$,
		 $\bx_2^0=(-0.3,0.4)$,
		 $\bx_3^0=(-0.4,-0.2)$;
case VII. $M=4$, $n_1=n_2=n_3=n_4=1$,
		  $\bx_1^0=(0.5,0)$,
		  $\bx_2^0=(0,0.5)$,
		  $\bx_3^0=(-0.5,0)$,
		  $\bx_4^0=(0,-0.5)$;
case VIII. $M=4$, $n_1=n_3=1$, $n_2=n_4=-1$,
		  $\bx_1^0=(0.5,0)$,
		  $\bx_2^0=(0,0.5)$,
		  $\bx_3^0=(-0.5,0)$,
		  $\bx_4^0=(0,-0.5)$;
case IX.  $M=4$, $n_2=n_3=-1$, $n_1=n_4=1$,
		  $\bx_1^0=(0.5,0)$,
		  $\bx_2^0=(0,0.5)$,
		  $\bx_3^0=(-0.5,0)$,
		  $\bx_4^0=(0,-0.5)$;	
case X.  $M=4$, $n_1=n_3=1$, $n_2=n_4=-1$,
		  $\bx_1^0=(0.5,0.5)$,
		  $\bx_2^0=(-0.5,0.5)$,
		  $\bx_3^0=(-0.5,-0.5)$,
		  $\bx_4^0=(0.5,-0.5)$;		
case XI.  $M=4$, $n_2=n_3=-1$, $n_1=n_4=1$,
		  $\bx_1^0=(0.5,0.5)$,
		  $\bx_2^0=(-0.5,0.5)$,
		  $\bx_3^0=(-0.5,-0.5)$,
		  $\bx_4^0=(0.5,-0.5)$;		  		  	
case XII.  $M=4$, $n_1=n_3=-1$, $n_2=n_4=1$,
		  $\bx_1^0=(0.4,0)$,
		  $\bx_2^0=(-0.4/3,0)$,
		  $\bx_3^0=(0.4/3,0)$,
		  $\bx_4^0=(0.4,0)$;			  		
case XIII.  $M=4$, $n_2=n_3=-1$, $n_1=n_4=1$,
		  $\bx_1^0=(0.4,0)$,
		  $\bx_2^0=(-0.4/3,0)$,
		  $\bx_3^0=(0.4/3,0)$,
		  $\bx_4^0=(0.4,0)$;			  		
case XIV.  $M=4$, $n_1=n_2=-1$, $n_3=n_4=1$,
		  $\bx_1^0=(0.4,0)$,
		  $\bx_2^0=(-0.4/3,0)$,
		  $\bx_3^0=(0.4/3,0)$,
		  $\bx_4^0=(0.4,0)$;			  		
case XV.  $M=4$, $n_1=n_3=-1$, $n_2=n_4=1$,
		 $\bx_1^0=(0.2,0.3)$,
		 $\bx_2^0=(-0.3,0.4)$,
		 $\bx_3^0=(-0.4,-0.2)$;	
		 $\bx_4^0=(0.3,-0.3)$.

\begin{figure}[htbp]
\centerline{\psfig{figure=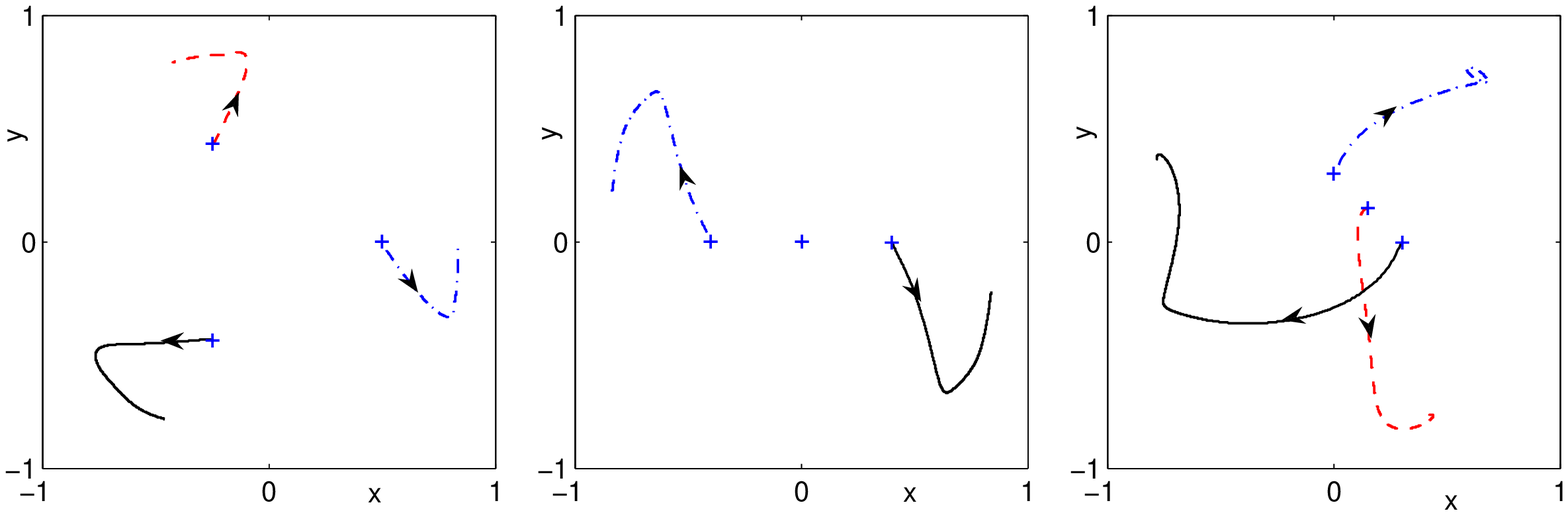,height=4cm,width=12.5cm,angle=0}}
\centerline{\psfig{figure=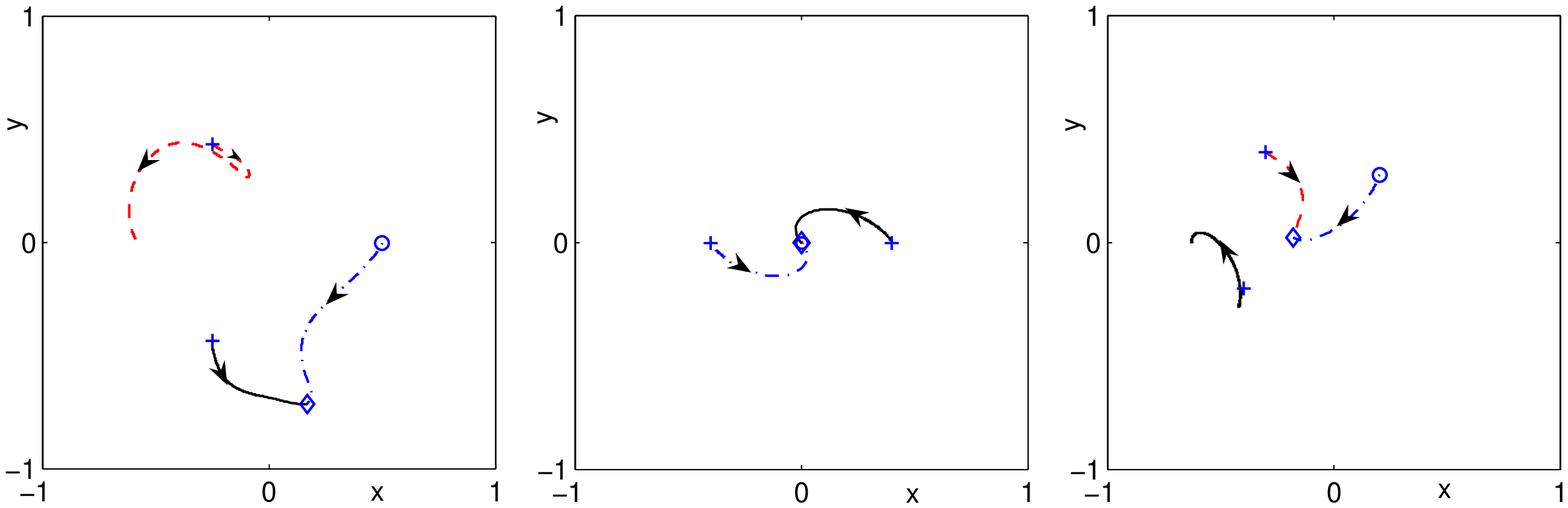,height=4cm,width=12.5cm,angle=0}}
\centerline{\psfig{figure=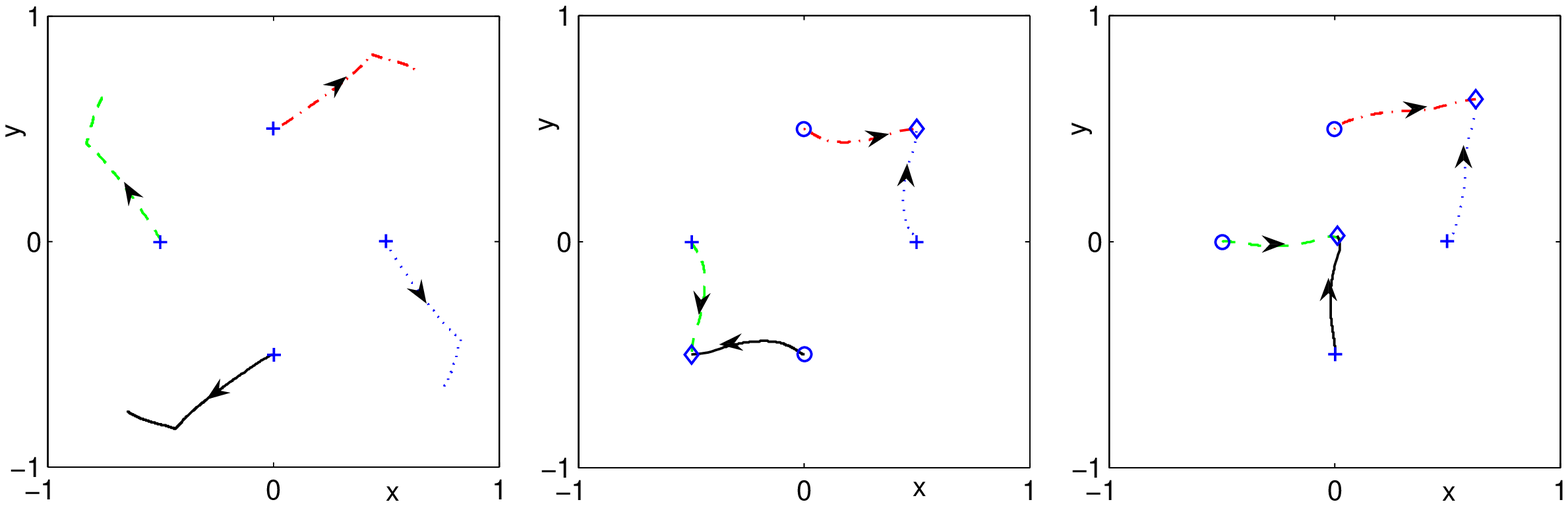,height=4cm,width=12.5cm,angle=0}}
\centerline{\psfig{figure=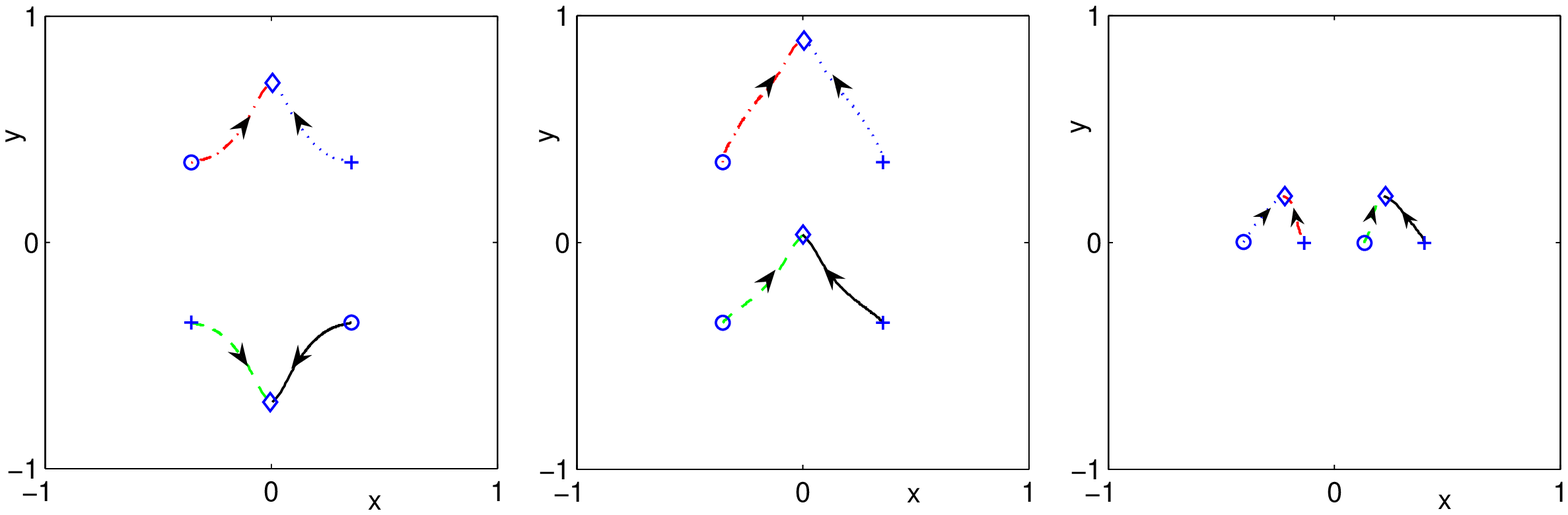,height=4cm,width=12.5cm,angle=0}}
\centerline{\psfig{figure=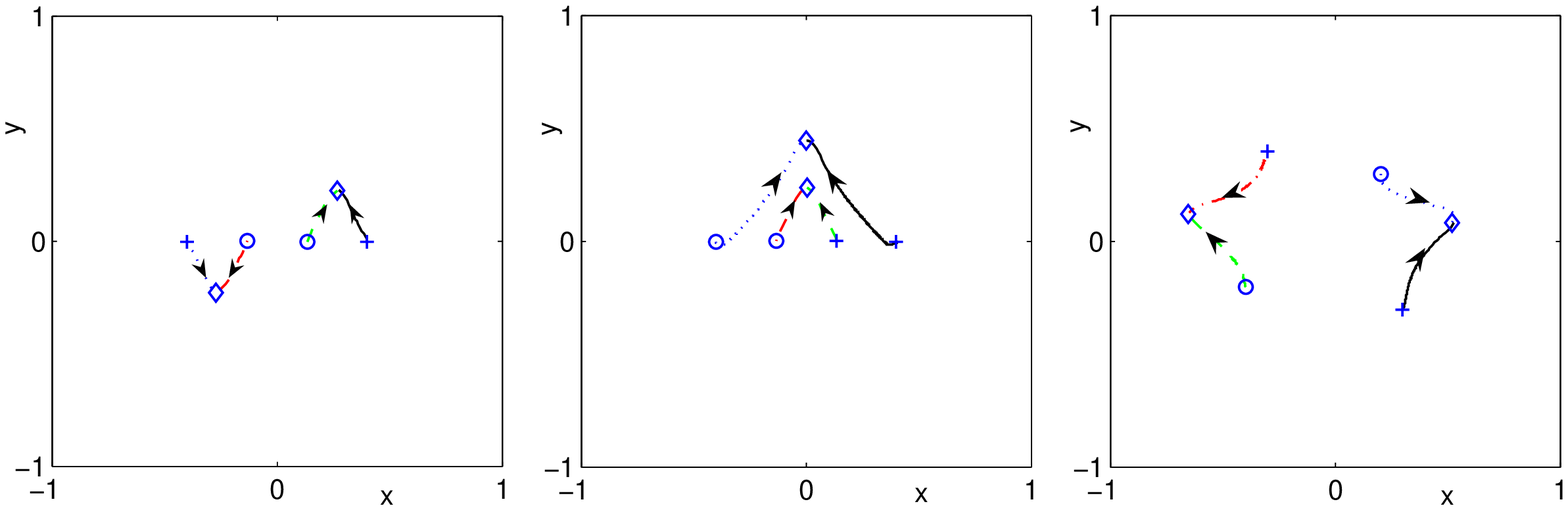,height=4cm,width=12.5cm,angle=0}}
 \caption{ Trajectory of vortex centers  for the interaction of
different vortex lattices in GLE under Dirichlet BC
with $\vep=\frac{1}{32}$ and $h(\bx)=0$ for  cases I-IX (from left to right and then from top to bottom)
in section \ref{sec: cgle_NRfVLD}. }
\label{fig: cgle-vortex-lattice-D}
 \end{figure}

Fig. \ref{fig: cgle-vortex-lattice-D} shows the trajectory of the vortex centers when $\vep=\frac{1}{32}$
in CGLE (\ref{cgle}) and $h(\bx)=0$ in (\ref{initdbc0}) for the above 15 cases.
From Fig. \ref{fig: cgle-vortex-lattice-D} and ample numerical experiments (not shown here for brevity),
we made the following observations:
(i). The dynamics and interaction of vortex lattices under the CGLE dynamics with Dirichlet BC
depends on its initial alignment of the lattice, geometry of the domain $\mathcal{D}$ and
the boundary value $g(\bx)$.
(ii). For a lattice of $M$ vortices, if they have the same index, then no collisions will
happen for any time $t\ge0$. On the other hand, if they have opposite index, e.g. $M^+>0$ vortices with index `$+1$' and $M^->0$ vortices with
index `$-1$' satisfying $M^++M^-=M$, collisions will always happen at finite time. More precisely,
when $t$ is sufficiently large, there exist exactly $|M^+-M^-|$ vortices with winding number `$+1$' if
$M^+> M^-$; while if $M^+< M^-$, there exist exactly $|M^+-M^-|$ vortices with winding number `$-1$'.

\begin{figure}[t!]
\centerline{(a)\;\psfig{figure=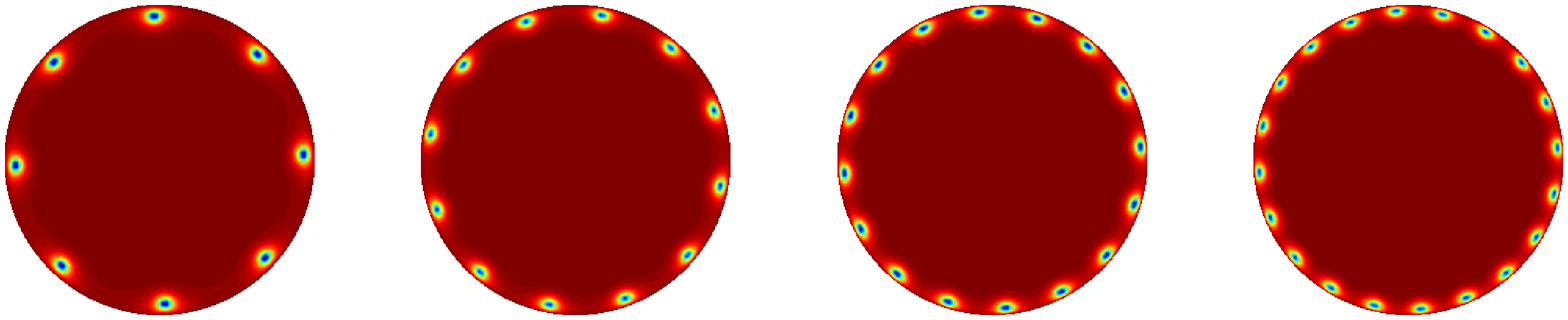,height=3.0cm,width=13.7cm,angle=0}}
\centerline{(b)\;\psfig{figure=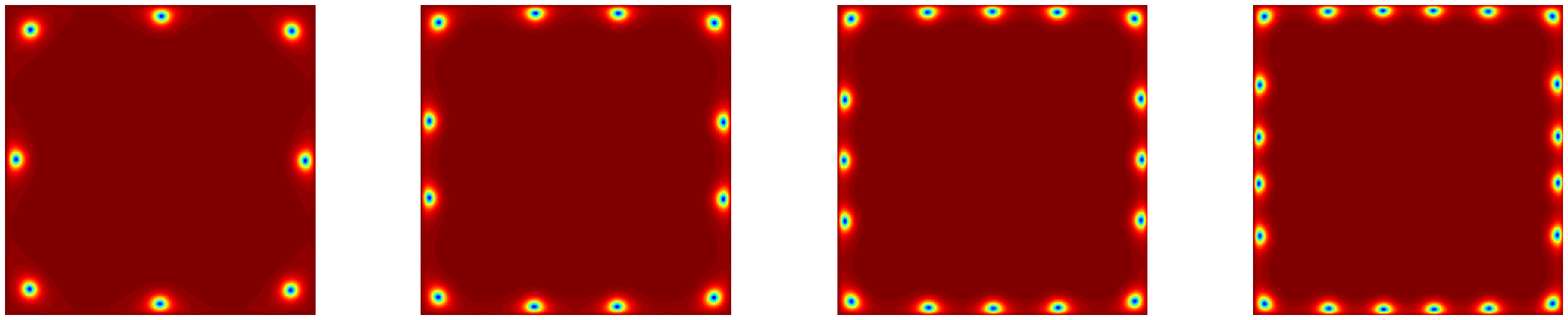,height=3.0cm,width=13.7cm,angle=0}}
\centerline{(c)\quad\psfig{figure=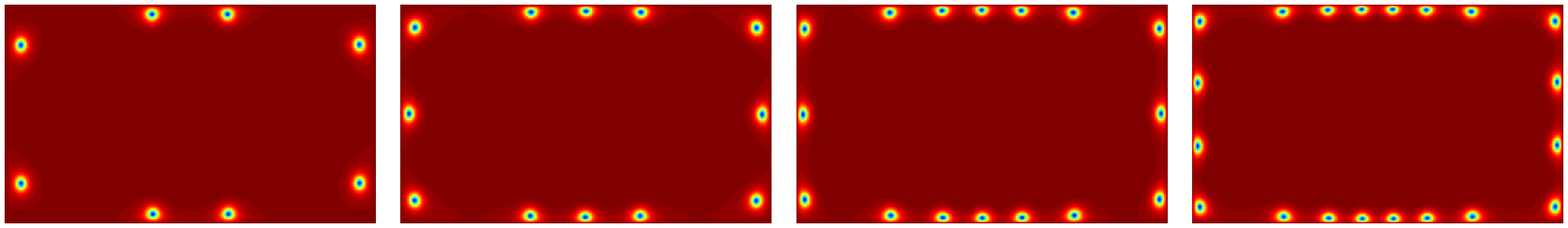,height=3.0cm,width=13.7cm,angle=0}}
 \caption{Contour plots of $|\phi^\vep(\bx)|$ for the steady states of vortex
lattice in CGLE under Dirichlet BC with $\vep=\frac{1}{32}$ for $M=8,\; 12,\; 16,\; 20$ (from left column to right column)
and different domains:
(a) unit disk $\mathcal{D}=B_1({\bf 0})$, (b) square domain $\mathcal{D}=[-1,1]^2$, (c) rectangular domain $\mathcal{D}=[-1.6,1.6]\times[-0.8,0.8]$.}
\label{fig: cgle-lattice_hx3_diff_domain}
\end{figure}


In order to study how the geometry of the domain and boundary conditions
take effect on the alignment of vortices in the steady state patterns
in the CGLE dynamics under Dirichlet BC, we made the following set-up
for our numerical computations. We chose the parameters as $\vep=\frac{1}{32}$,
$$n_j=1,\qquad \bx_j^{0}=0.5\left(\cos\left(\frac{2j\pi}{M}\right),
\sin\left(\frac{2j\pi}{M}\right)\right),\qquad j=1,2,\ldots,M,$$
i.e., initially we have $M$
like vortices which are located uniformly in a circle centered at origin with radius
$R_1=0.5$.
Denote $\phi^\vep(\bx)$ as the steady state, i.e.,
$\phi^\vep(\bx) =\lim_{t\to\infty}\psi^\vep(\bx,t)$ for
$\bx\in \mathcal{D}$. Fig. \ref{fig: cgle-lattice_hx3_diff_domain} depicts the
contour plots of the amplitude $|\phi^\vep|$ of the steady
state in the CGLE dynamics with $h(\bx)=0$ in
(\ref{initdbc0}) for different $M$ and domains,
and Fig. \ref{fig: cgle-lattice_den_disk_square_N12_diffhx} depicts
similar results with $M=12$ for different  $h(\bx)$ in (\ref{initdbc0}).

\begin{figure}[t!]
\centerline{\psfig{figure=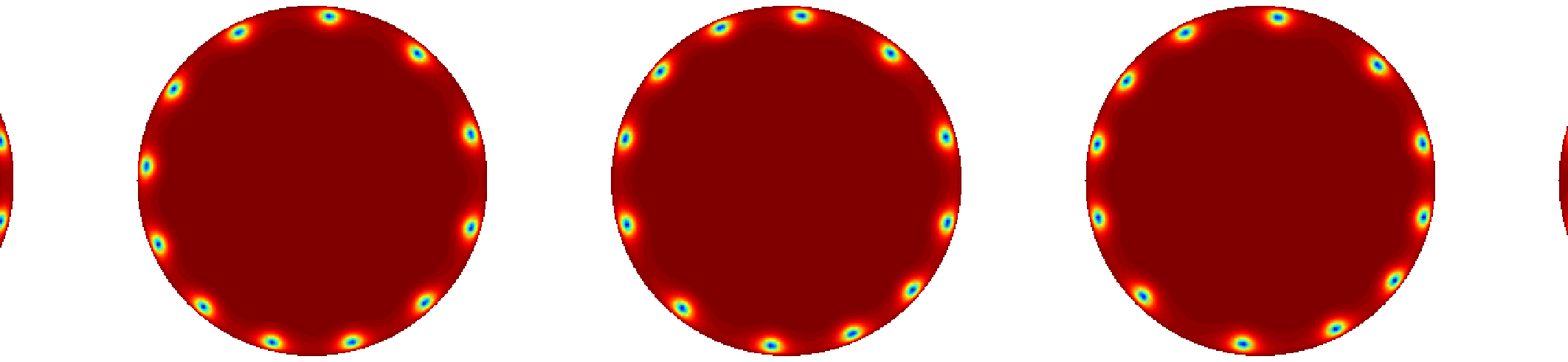,height=2.80cm,width=15cm,angle=0}}
\centerline{\psfig{figure=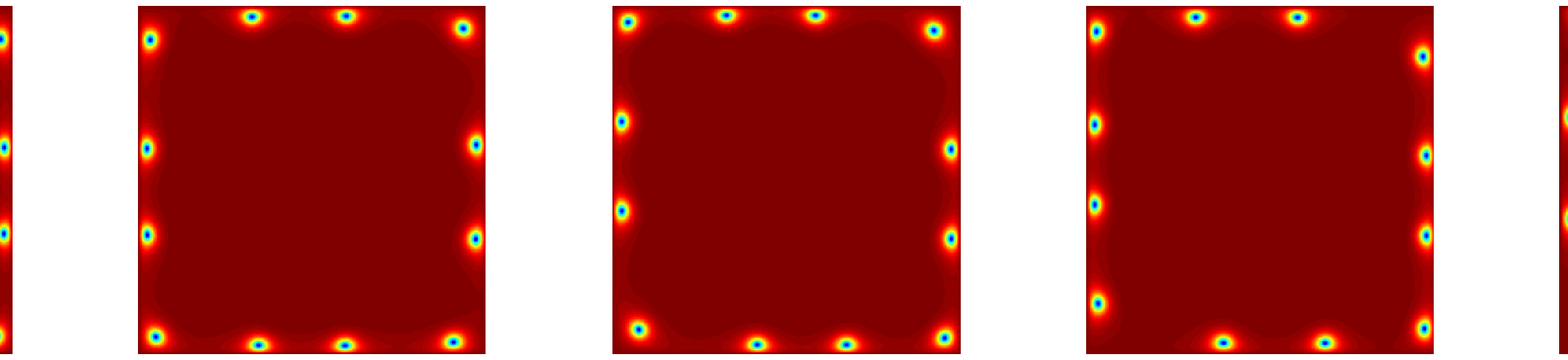,height=2.80cm,width=15cm,angle=0}}
\centerline{\psfig{figure=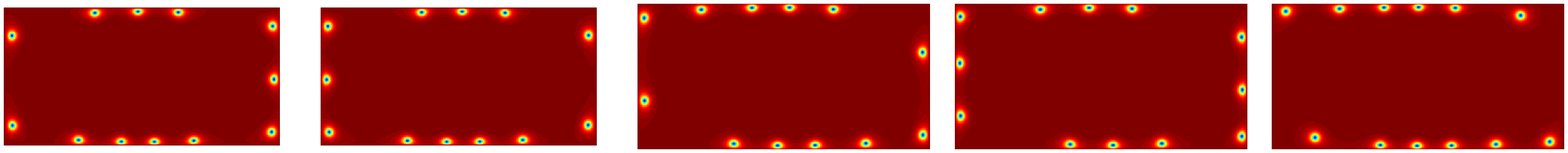,height=3.30cm,width=15cm,angle=0}}
\caption{Contour plots of $|\phi^\vep(\bx)|$ for the steady states of vortex
lattice in CGLE under Dirichlet BC with $\vep=\frac{1}{32}$ and
$M=12$ on a unit disk $\mathcal{D}=B_1({\bf 0})$ (top row) or a square $\mathcal{D}=[-1,1]^2$ (middle row) or a
rectangular domain $\mathcal{D}=[-1.6,1.6]\times[-0.8,0.8]$ (bottom row)
under different $h(\bx)=x+y, \; x^2-y^2,\ x-y,\; x^2-y^2+2xy,\;x^2-y^2-2xy$ (from left column to right column).}
\label{fig: cgle-lattice_den_disk_square_N12_diffhx}
\end{figure}

Based on Figs. \ref{fig: cgle-lattice_hx3_diff_domain}-\ref{fig: cgle-lattice_den_disk_square_N12_diffhx}
and ample numerical results (not shown here for brevity),
we made the following observations for
the steady state patterns of vortex lattices under the CGLE dynamics
with Dirichlet BC:
(i). The vortex lattices with the same winding number undergo repulsive interaction between each other and finally they move to
somewhere near the boundary of the domain. During the evolution process,
no particle-like collision phenomena happen and a
steady state pattern is finally formed when $t\to\infty$. As a matter of fact, the
steady state is also the solution of the following minimization problem
\[\phi^\vep =\displaystyle {\rm argmin}_{\phi(\bx)|_{\bx\in\partial\mathcal{D}}=
\psi_0^\vep(\bx)|_{\bx\in\partial\mathcal{D}} } \mathcal{E}^\vep(\phi).\]
(ii). Both the geometry of the domain and the phase shift,
i.e. $h(\bx)$, will take significant effect on steady state patterns.
(iii). At the steady state, the distance between the vortex centers
 and the boundary of the domain depends on $\vep$ and $M$.
 If $M$ is fixed, when $\vep$ decreases,
 the distance decreases; while if $\vep$ is fixed,
 when $M$ increases, the distance decreases. We remark it here as
 an interesting open problem to find how the width depends on the
 value of $\vep$, the boundary condition as well as the geometry of
 the domain.

\section{Numerical results under Neumann BC}
\label{sec: cgle_Neu}

In this section, we report numerical results
for vortex interactions of the CGLE (\ref{cgle}) under
the homogeneous Neumann BC (\ref{neu-bc}) and compare
them with those obtained from the corresponding RDLs.
The initial condition $\psi_0^\vep$ in (\ref{ini_con}) is also
chosen as the form (\ref{initdbc0}), but with harmonic function $h(\bx)$ replaced as
$h_n(\bx)$ so that it will satisfy the Neumann BC as
\[\frac{\partial}{\partial \bn}h_n(\bx) =
-\frac{\partial}{\partial \bn}\sum_{l=1}^{M}
n_{l}\theta(\bx-\bx_{l}^0), \quad \bx\in\partial \Omega.\]
The CGLE (\ref{cgle}) with (\ref{neu-bc})  and
(\ref{initdbc0}) is solved by the numerical method TSCP presented in section \ref{sec: num_method}
in the following simulations.

\subsection{Single vortex}
\label{sec: cgle_NRfSVN}

In the subsection, we present numerical results of the motion of a single quantized vortex in
the CGLE dynamics with Neumann BC and its corresponding reduced dynamical
laws. We choose the parameters as $M=1$ and $n_1=1$ in (\ref{initdbc0}). Fig. \ref{fig: cgle-one-vortex-N}
shows the trajectory of the vortex center for different $\bx_1^0$ in (\ref{initdbc0})
when $\vep=\frac{1}{25}$ as well as time evolution of $\bx_1^{\vep}$ and $d_1^\vep$ for different $\vep$.

By observing Fig. \ref{fig: cgle-one-vortex-N} and ample numerical simulation results
(not shown here for brevity), we could see that: (i). The initial location of the vortex affects
the motion of the vortex significantly and this reflects the boundary effect coming from the Neumann BC.
(ii). If $\bx_1^0=(0,0)$, the vortex does not move at any time;
otherwise, the vortex does move and it will run out of the domain and
never come back. This phenomenon is quite different from the case
with Dirichlet BC in bounded domains, in which a single vortex can
never move out of the domain, or the case with the initial condition (\ref{initdbc0}) in the whole space,
in which a single vortex doesn't move at all, regardless of the initial location of the vortex for the both
cases.
(iii). As $\vep\rightarrow 0$, the dynamics of the vortex center
under the CGLE dynamics converges uniformly in time to
that of the RDLs very well before it exits the domain,
which verifies numerically the validation of the RDLs
in this case. Apparently, when the vortex center moves out of the domain,
the reduced dynamics laws are no longer valid.
Based on our extensive numerical experiments,
the motion of the vortex center from the RDLs
agrees with that from the CGLE dynamics qualitatively when $0<\vep<1$
and quantitatively when $0<\vep\ll 1$ before it moves out of the domain.

\begin{figure}[t!]
 \centerline{(a)\psfig{figure=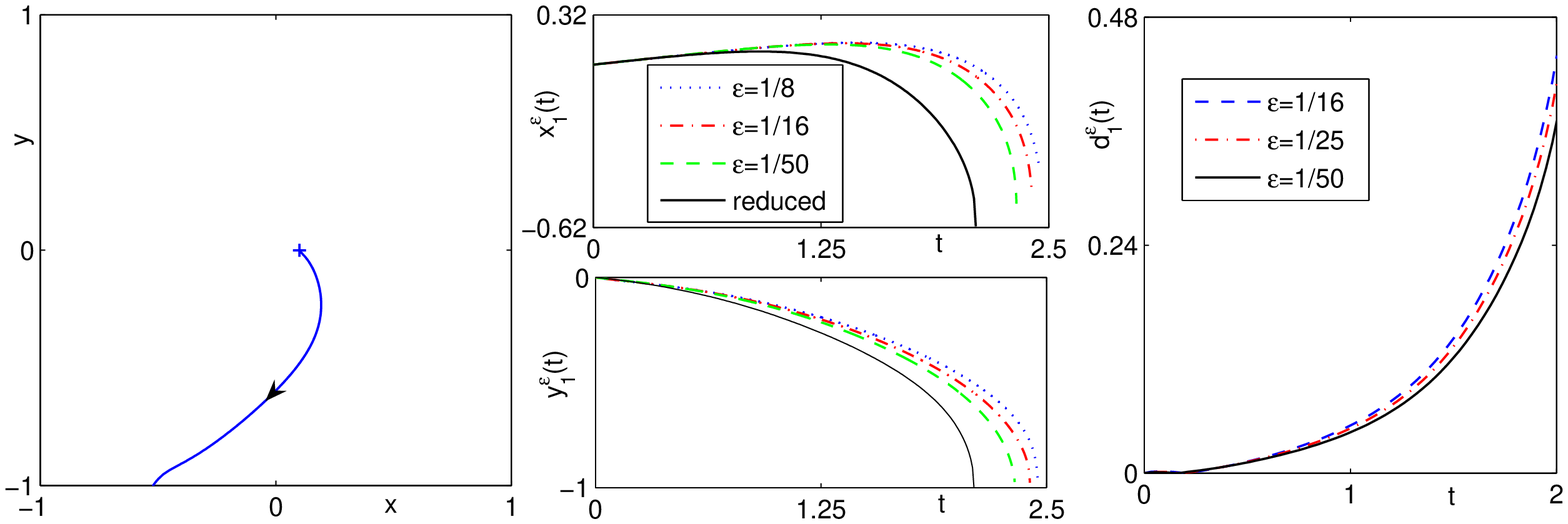,height=4cm,width=12.5cm,angle=0}}
 \centerline{(b)\psfig{figure=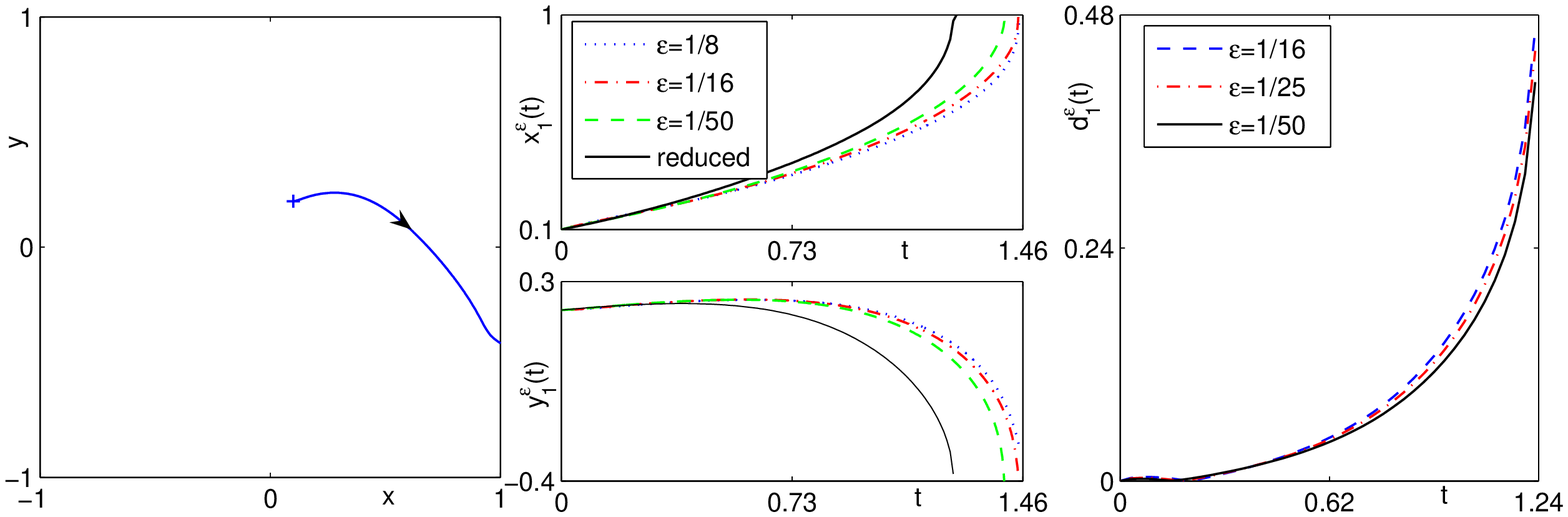,height=4cm,width=12.5cm,angle=0}}
 \caption{Trajectory of the vortex center when $\vep=\frac{1}{25}$ (left) as well as
 time evolution of $\bx_1^{\vep}$ (middle)  and $d_1^\vep$  for different $\vep$ (right) for the motion  of
a single vortex in CGLE under homogeneous Neumann BC with different
 $\bx_1^0$ in (\ref{initdbc0}) in section \ref{sec: cgle_NRfSVN}: (a) $\bx_1^0=(0.1,0)$, (b) $\bx_1^0=(0.1,0.2)$.}
\label{fig: cgle-one-vortex-N}
\end{figure}

\begin{figure}[t!]
\centerline{ (a)
 \psfig{figure=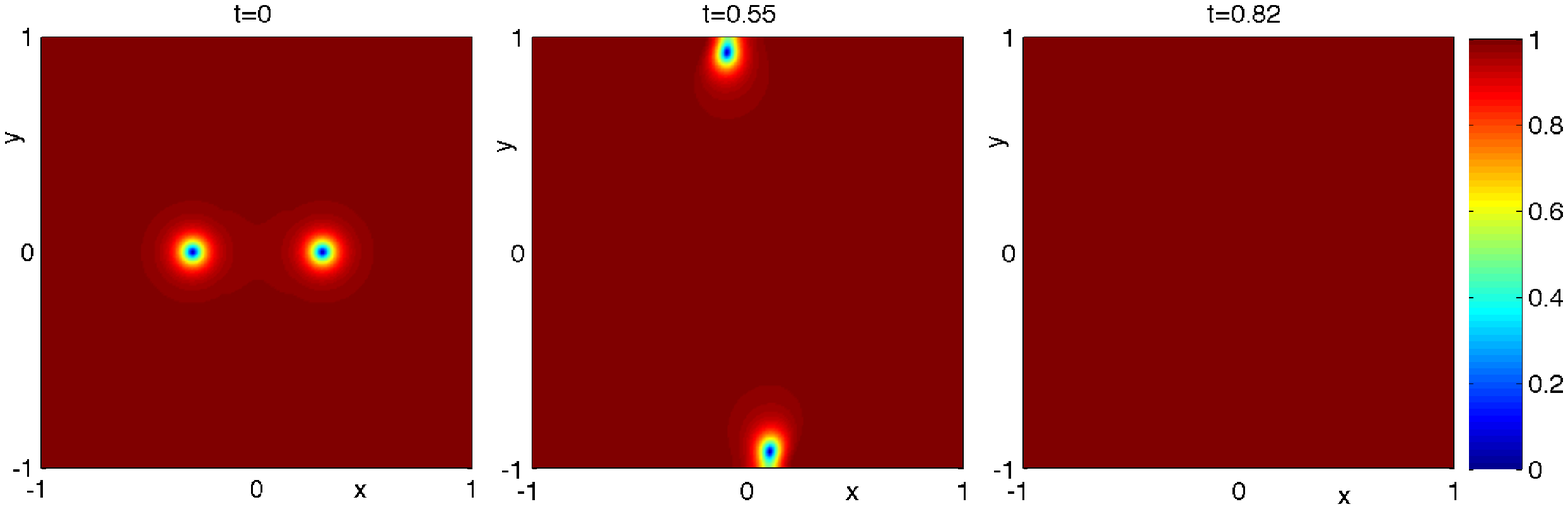,height=3.5cm,width=10.5cm,angle=0}
 (c)\psfig{figure=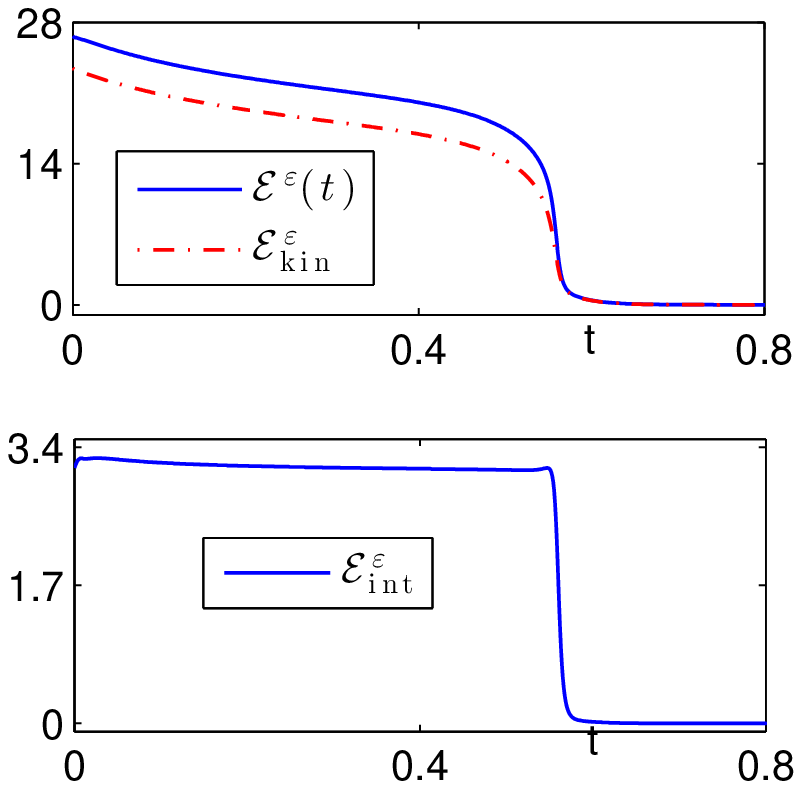,height=3.4cm,width=4cm,angle=0}
 }
\centerline{ (b)
\psfig{figure=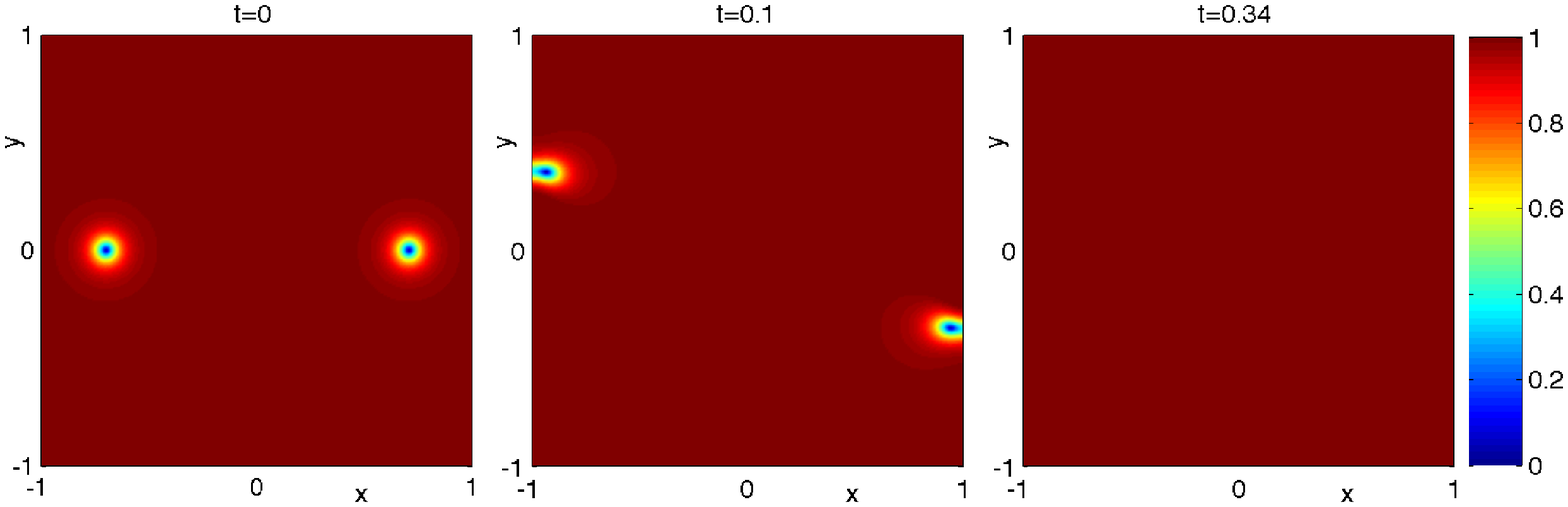,height=3.5cm,width=10.5cm,angle=0}
(d)\psfig{figure=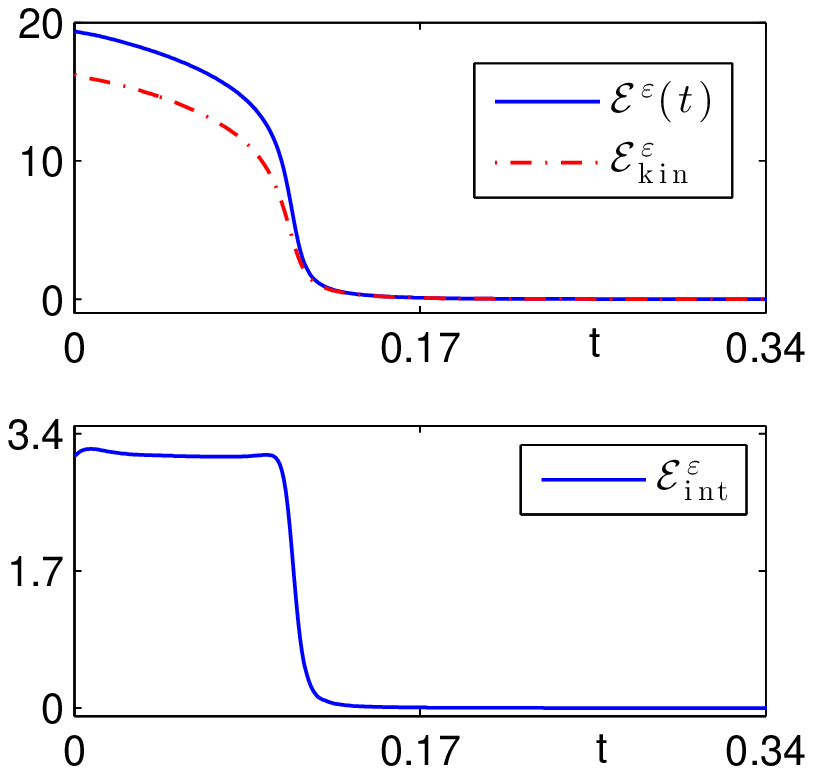,height=3.4cm,width=4cm,angle=0}
}
\caption{ Contour plots of $|\psi^{\vep}(\bx,t)|$ at different times
when $\vep=\frac{1}{25}$ ((a) \& (b)) and the corresponding time
evolution of the GL functionals ((c) \& (d)) for the motion  of vortex pair
in CGLE under homogeneous Neumann BC with different
 $d_0$ in (\ref{initdbc0}) in section \ref{sec: cgle_NRfVPN}:  top row:
 $d_0=0.3$, bottom row: $d_0=0.7$.   } \label{fig: cgle-vortex-pair-N-dens}
\end{figure}

\subsection{Vortex pair}
\label{sec: cgle_NRfVPN}

Here we present numerical results of the interaction of vortex pair
under the CGLE dynamics with Neumann BC and its corresponding reduced
dynamical laws. We choose the simulation parameters as $M=2$, $n_1=n_2=1$ and
$\bx_2^0=-\bx_1^0=(d_0,0)$
 with $0<d_0<1$ in (\ref{initdbc0}). Fig.
 \ref{fig: cgle-vortex-pair-N-dens} shows the contour plots of
 $|\psi^{\vep}(\bx,t)|$ at different times
when $\vep=\frac{1}{25}$, and Fig.
\ref{fig: cgle-vortex-pair-N} shows the trajectory of the vortex
pair when $\vep=\frac{1}{25}$ as well as time evolution of
$x_1^{\vep}(t)$ and  $d_1^{\vep}(t)$
for different $d_0$ in (\ref{initdbc0}).

\begin{figure}[t!]
 \centerline{
(a)\psfig{figure=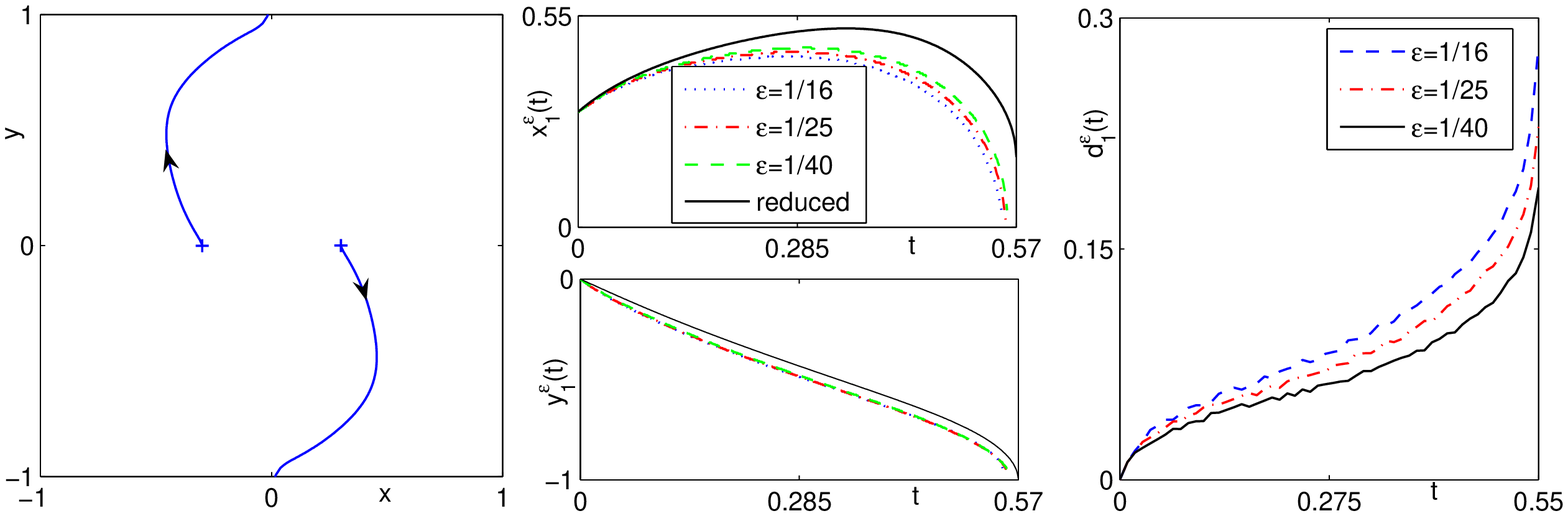,height=4cm,width=12.5cm,angle=0}}
 \centerline{
(b)\psfig{figure=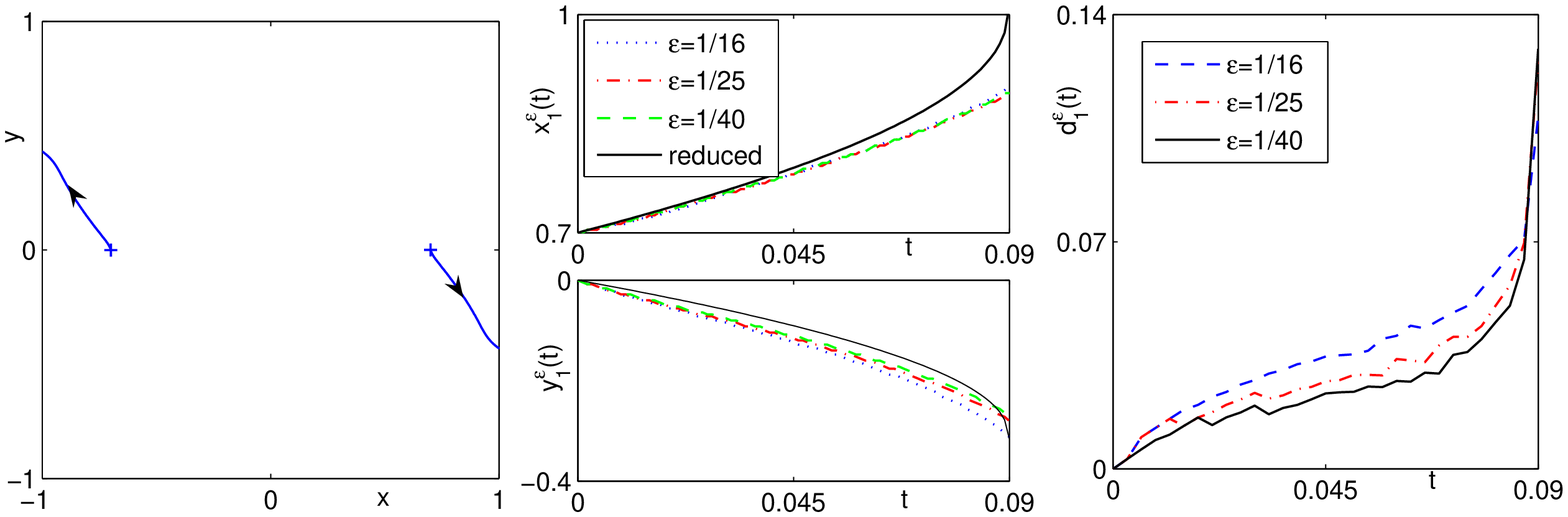,height=4cm,width=12.5cm,angle=0}}
\caption{ Trajectory of the vortex center when $\vep=\frac{1}{25}$ (left) as well as
 time evolution of $\bx_1^{\vep}$ (middle) and $d_1^\vep$  for different $\vep$ (right) for the motion  of
 vortex pair in CGLE under homogeneous Neumann BC with different
 $d_0$ in (\ref{initdbc0}) in section \ref{sec: cgle_NRfVPN}: (a) $d_0=0.3$, (b) $d_0=0.7$.}
 \label{fig: cgle-vortex-pair-N}
\end{figure}

From Figs. \ref{fig: cgle-vortex-pair-N-dens}-\ref{fig: cgle-vortex-pair-N}
and ample numerical results (not shown here for brevity),
we made the following observations: (i). The initial location of the vortex, i.e., the value of $d_0$ affects
the motion of the vortex significantly and this reflects the boundary effect coming from the Neumann BC.
(ii).  For the CGLE with $\vep$ fixed, there exists a sequence of critical values
$d_1^{c,\vep}>d_2^{c,\vep}>d_3^{c,\vep}>\cdots>d_k^{c,\vep}>\cdots$, which can
determine the escape approach about how the vortex pair moves out of the domain.
More precisely, if the value of $d_0$ falls into the interval $(d_{2n+1}^{c,\vep},d_{2n}^{c,\vep})$,
where $n=0,1,\ldots$, and $d_{0}^{c,\vep}=+\infty$, then the two vortices will move out of the domain from the side boundary;
otherwise, if it falls into the interval $(d_{2n+2}^{c,\vep},d_{2n+1}^{c,\vep})$,
they will move out of the domain from the top-bottom boundary.
For the RDL, there also exists such corresponding sequence of critical values
$\{d_k^{c,r}, k=0,1,\ldots\}$ which determine the trajectory of the vortex pair motion.
We note that it might be an interesting problem to find the values of those $d_k^{c,\vep}$ and $d_k^{c,r}$ and
study their convergence relations between them.
(iii). The motion of the vortex pair exhibits hybrid properties of that in the GLE dynamics and NLSE dynamics
with Neumann BC. As given by previous studies \cite{BT,BT1},
a vortex pair in the GLE dynamics will
always move outward along the line that connects with the two vortices and
finally they will move out of the domain, while in the NLSE dynamics, they will
always rotate around each other periodically.
Based on our extensive numerical results, we also found that under a fixed initial setup, the larger the value $\beta$ becomes,
the more rotations the vortex pair will do before they exit the domain, which means that as $\beta$ becomes larger,
the closer the motion in CGLE dynamics is to that in NLSE dynamics;  on the other hand, as the value $\alpha$ becomes larger,
the time when the vortex pair exits the domain becomes faster,
which also means that the motion in CGLE dynamics becomes closer to that in GLE dynamics.
This gives sufficient numerical evidence for our conclusion.
(iv). As $\vep\rightarrow 0$, the dynamics of the vortex pair under the CGLE dynamics converges uniformly in time to
that of the RDLs very well before either of the two vortices exit the domain,
which verifies numerically the validation of the RDLs
in this case.
(iv). During the dynamics evolution of CGLE, the GL functional
and its kinetic parts decrease as the time increases. They
do not change much when $t$ is small and change dramatically
when either of the two vortices move out of the domain. When $t\to \infty$, all the
three quantities converge to 0 (see Fig. \ref{fig: cgle-vortex-pair-N-dens} (c) \& (d)),
which indicates that a constant steady state have been reached in the form of  $\phi^\vep(\bx)=e^{ic_{_0}}$
for $\bx\in\mathcal{D}$ with $c_{_0}$ a constant.

\begin{figure}[t!]
\centerline{ (a)
 \psfig{figure=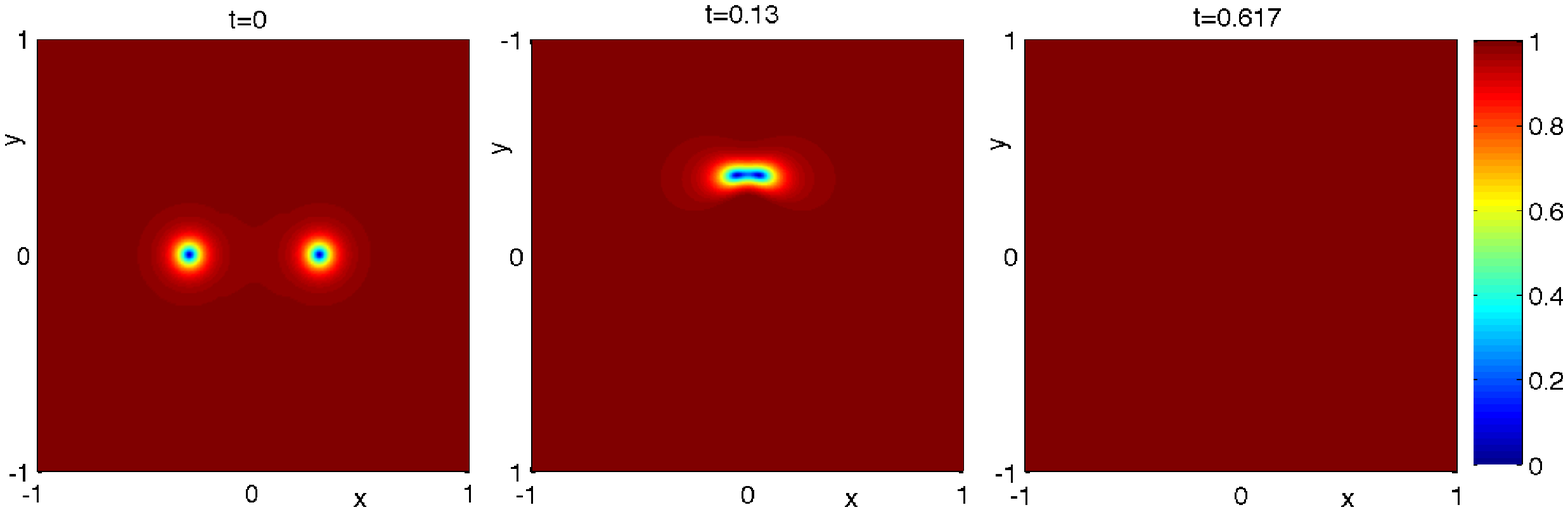,height=3.5cm,width=10.5cm,angle=0}
 (c)\psfig{figure=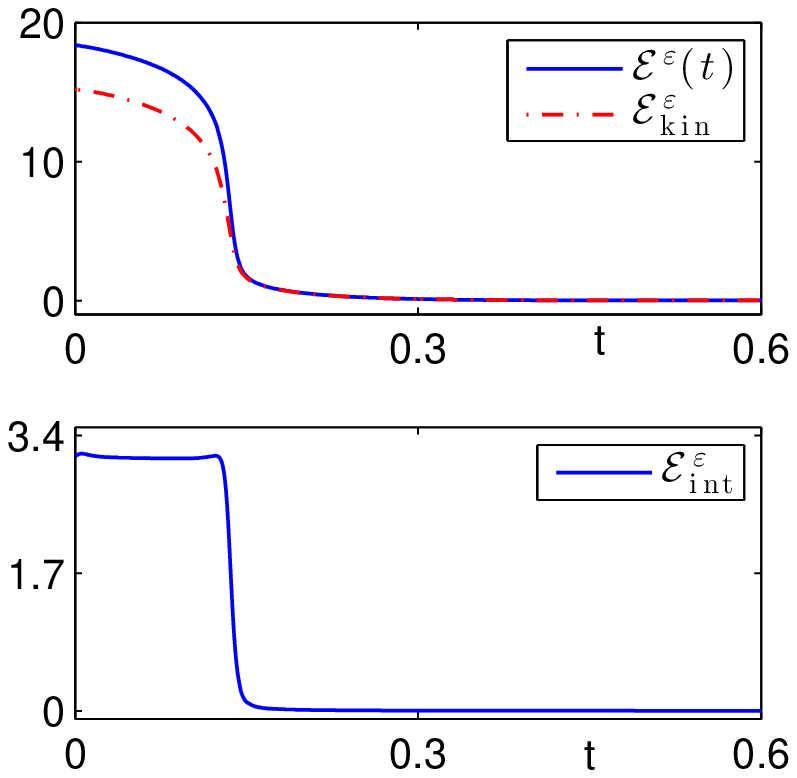,height=3.4cm,width=4cm,angle=0}
 }
\centerline{ (b)
\psfig{figure=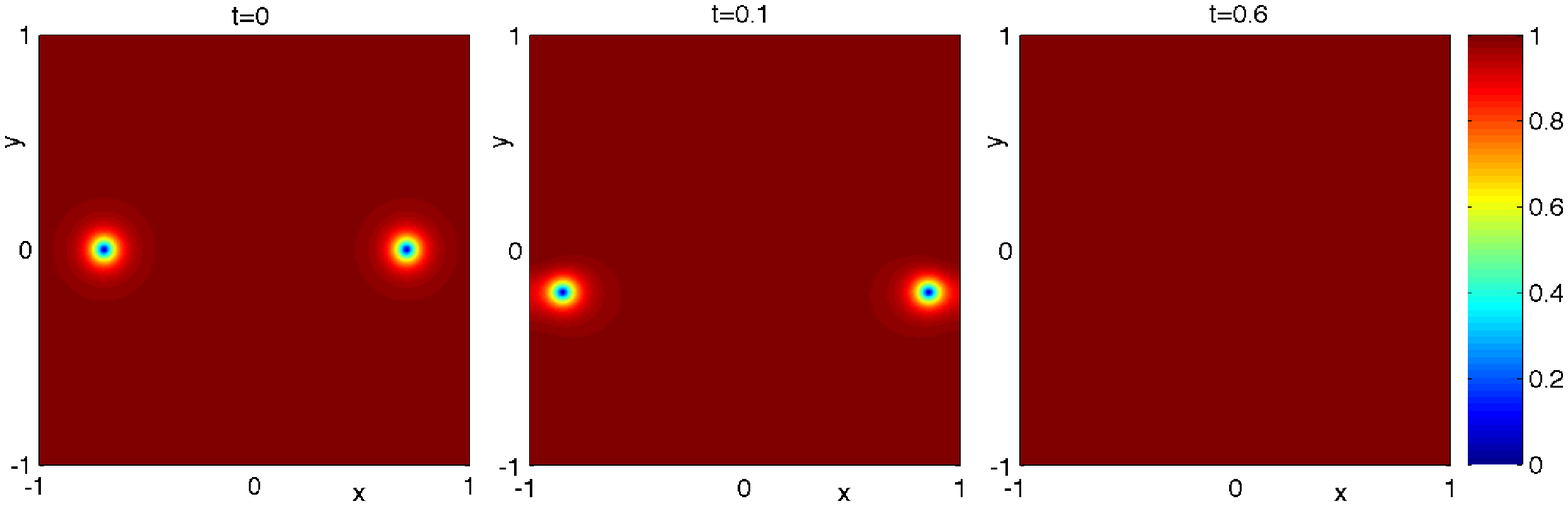,height=3.5cm,width=10.5cm,angle=0}
(d)\psfig{figure=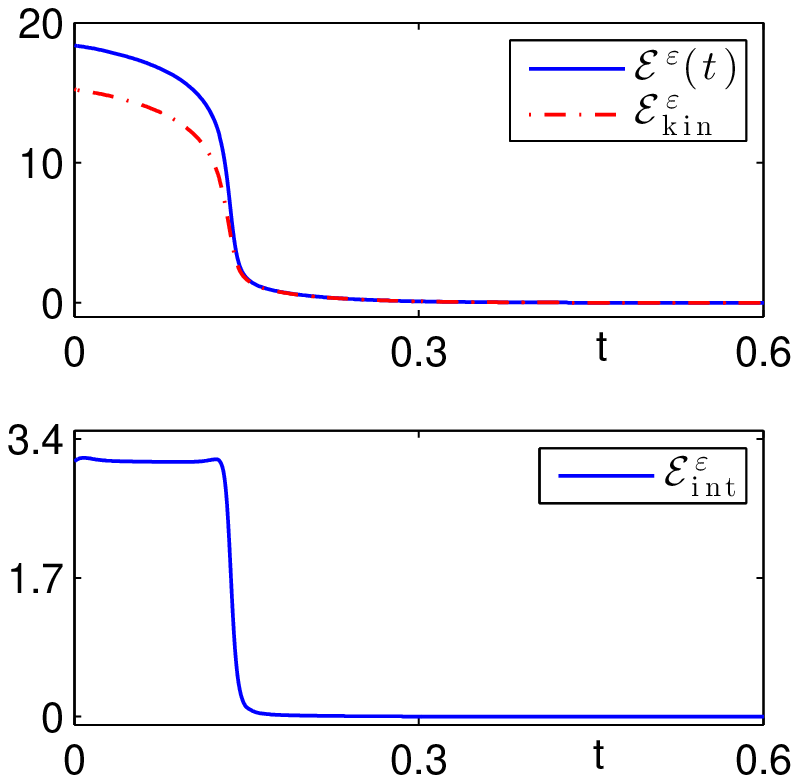,height=3.4cm,width=4cm,angle=0}
}
\caption{ Contour plots of $|\psi^{\vep}(\bx,t)|$ at different times
when $\vep=\frac{1}{25}$ and the corresponding time evolution of the GL functionals  for the motion  of
 vortex dipole in CGLE under homogeneous Neumann BC with different
 $d_0$ in (\ref{initdbc0}) in section \ref{sec: cgle_NRfVDN}:  top row: $d_0=0.3$, bottom row: $d_0=0.7$.   } \label{fig: cgle-vortex-dipole-N-dens}
\end{figure}

\subsection{Vortex dipole}
  \label{sec: cgle_NRfVDN}

Here we present numerical results of the interaction of vortex dipole
in the CGLE dynamics with Neumann BC and its corresponding reduced
dynamics. We choose the simulation parameters as $M=2$, $n_2=-n_1=1$ and
$\bx_2^0=-\bx_1^0=(d_0,0)$ with $0<d_0<1$ in (\ref{initdbc0}).
Fig. \ref{fig: cgle-vortex-dipole-N-dens} shows the contour plots of
 $|\psi^{\vep}(\bx,t)|$ at different times
when $\vep=\frac{1}{25}$, and Fig.
\ref{fig: cgle-vortex-dipole-N} depicts the trajectory of the vortex
pair when $\vep=\frac{1}{25}$ as well as time evolution of
$x_1^{\vep}(t)$ and  $d_1^{\vep}(t)$
for different $d_0$ in (\ref{initdbc0}).

\begin{figure}[t!]
 \centerline{
\psfig{figure=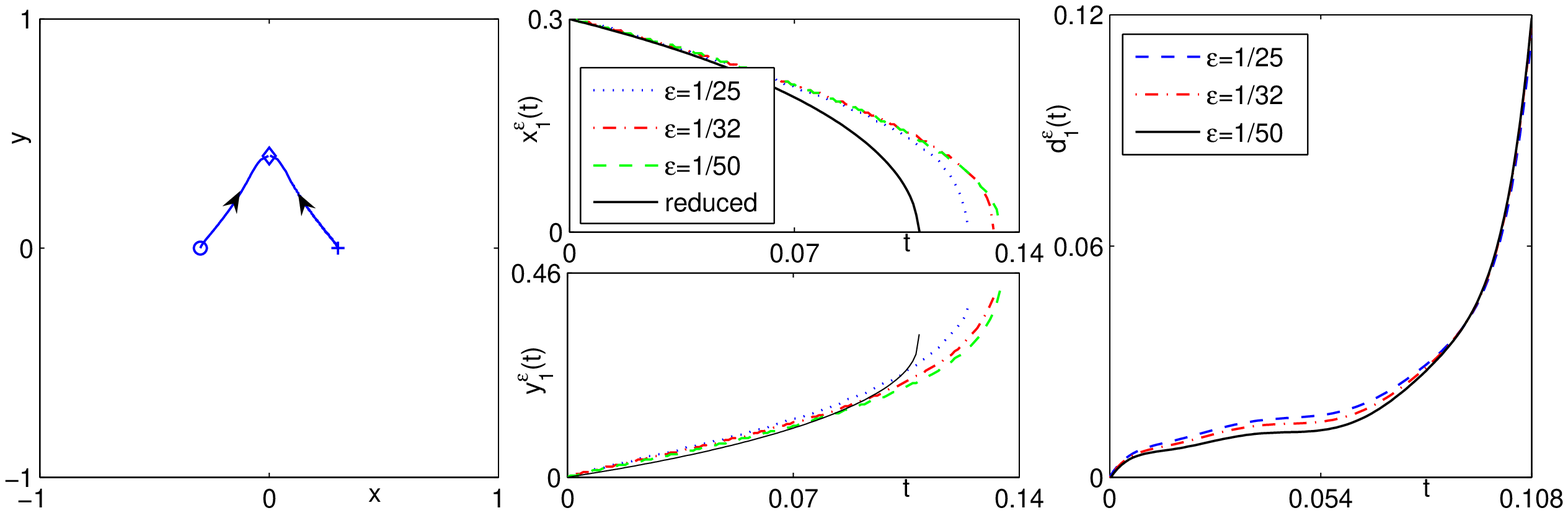,height=4cm,width=12.5cm,angle=0}}
 \centerline{
\psfig{figure=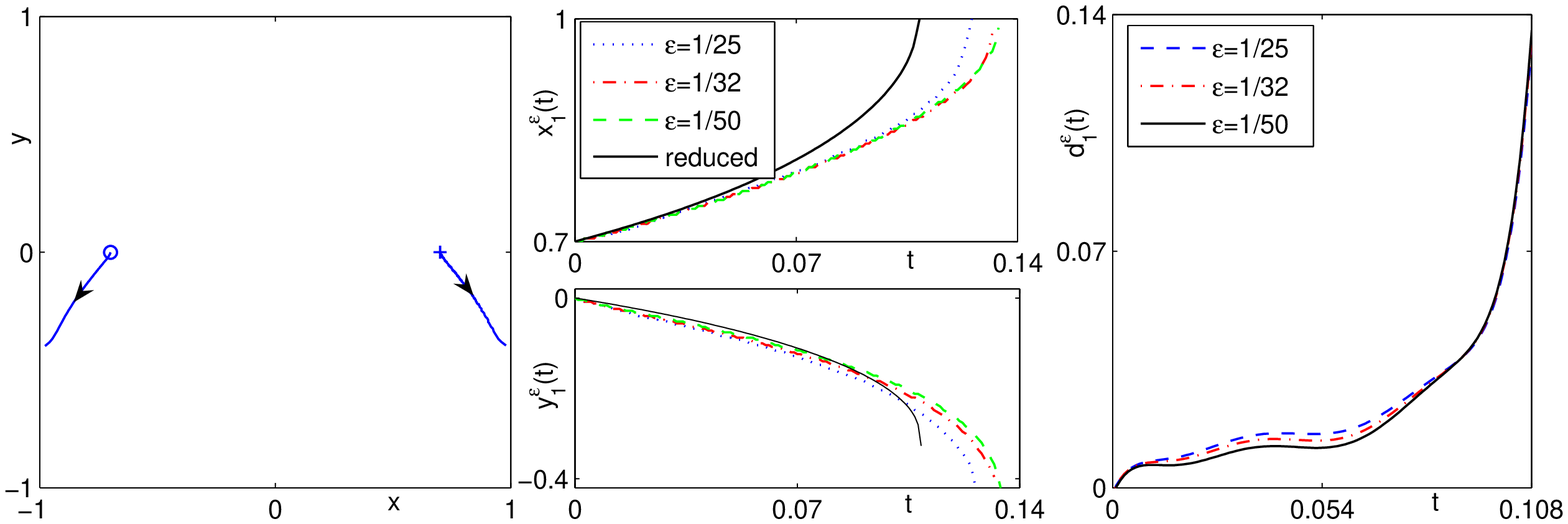,height=4cm,width=12.5cm,angle=0}}
\caption{ Trajectory of the vortex center when $\vep=\frac{1}{25}$ (left) as well as
 time evolution of $\bx_1^{\vep}$ (middle)  and $d_1^\vep$  for different $\vep$ (right) for the motion  of
 vortex dipole in CGLE under homogeneous Neumann BC with different
 $d_0$ in (\ref{initdbc0}) in section \ref{sec: cgle_NRfVPN}: (a) $d_0=0.3$, (b) $d_0=0.7$.   } \label{fig: cgle-vortex-dipole-N}
\end{figure}

From Fig. \ref{fig: cgle-vortex-dipole-N-dens} and \ref{fig: cgle-vortex-dipole-N}
and ample numerical
results (not shown here for brevity),
we can make the following observations for the interaction
of vortex pair under the NLSE dynamics with homogeneous Neumann
BC:
(i). The initial location of the vortex, i.e., the value of $d_0$ affects
the motion of the vortex significantly.
(ii). For the CGLE with $\vep$ fixed, there exists a critical value $d_c^{\vep}$
 such that: if $d_0>d_c^{\vep}$, the two vortices will exit the domain from the side boundary;
 otherwise, they will merge somewhere in the domain.
 For the RDL, there also exists such corresponding critical values $d_c^{r}$.
 We also note that it might
 be an interesting problem to find those values $d_c^{\vep}$ and  $d_c^{r}$,
 and to study their convergence relation.
(iii). As $\vep\rightarrow 0$, the dynamics of the two vortex centers under the CGLE dynamics converge uniformly in time to
that of the RDLs very well before they move out of the domain or merge with each other,
which verifies numerically the validation of the RDLs in this case.

\begin{figure}[htbp]
\centerline{\psfig{figure=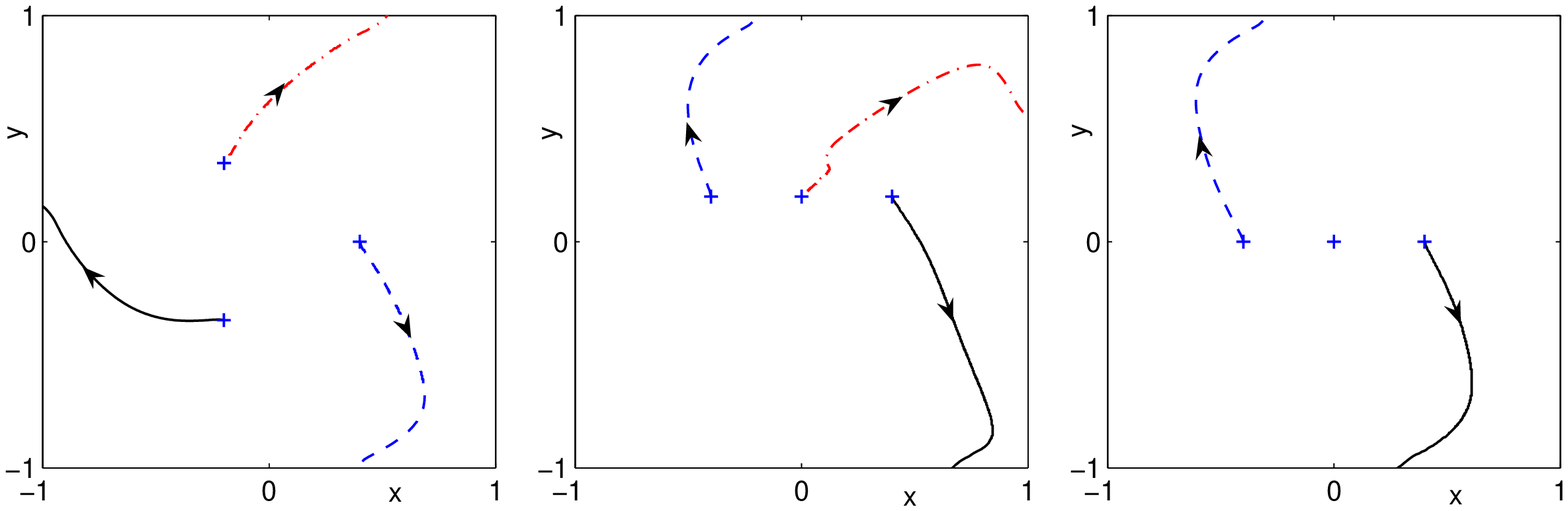,height=4cm,width=12.5cm,angle=0}}
\centerline{\psfig{figure=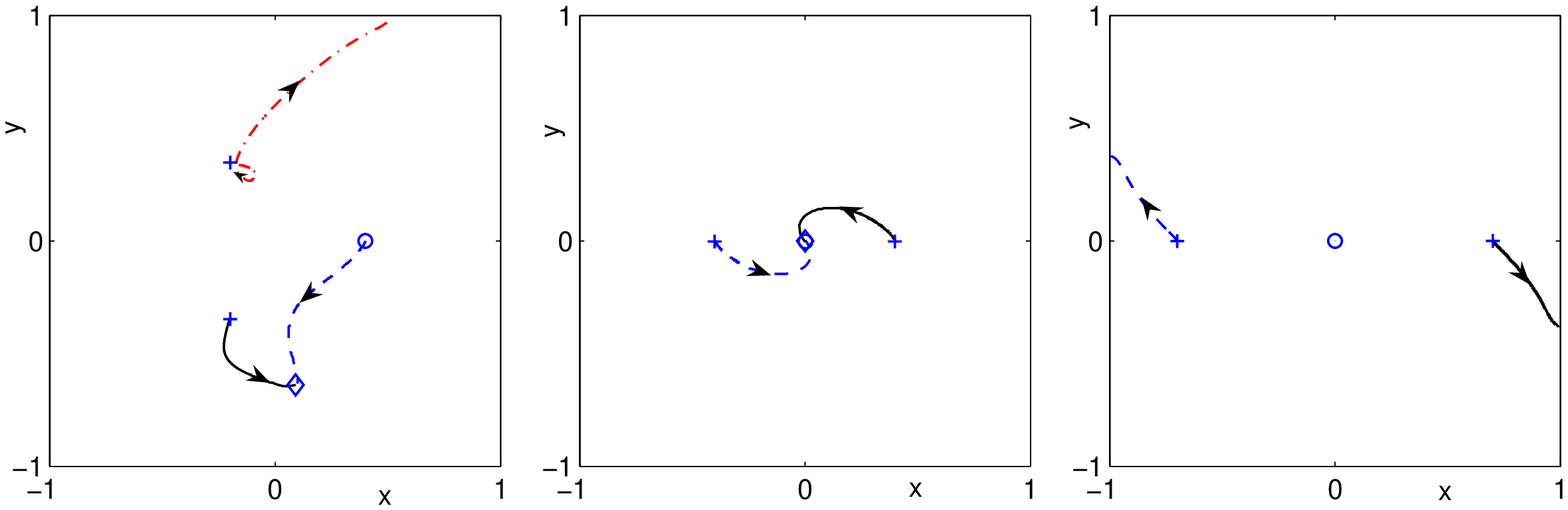,height=4cm,width=12.5cm,angle=0}}
\centerline{\psfig{figure=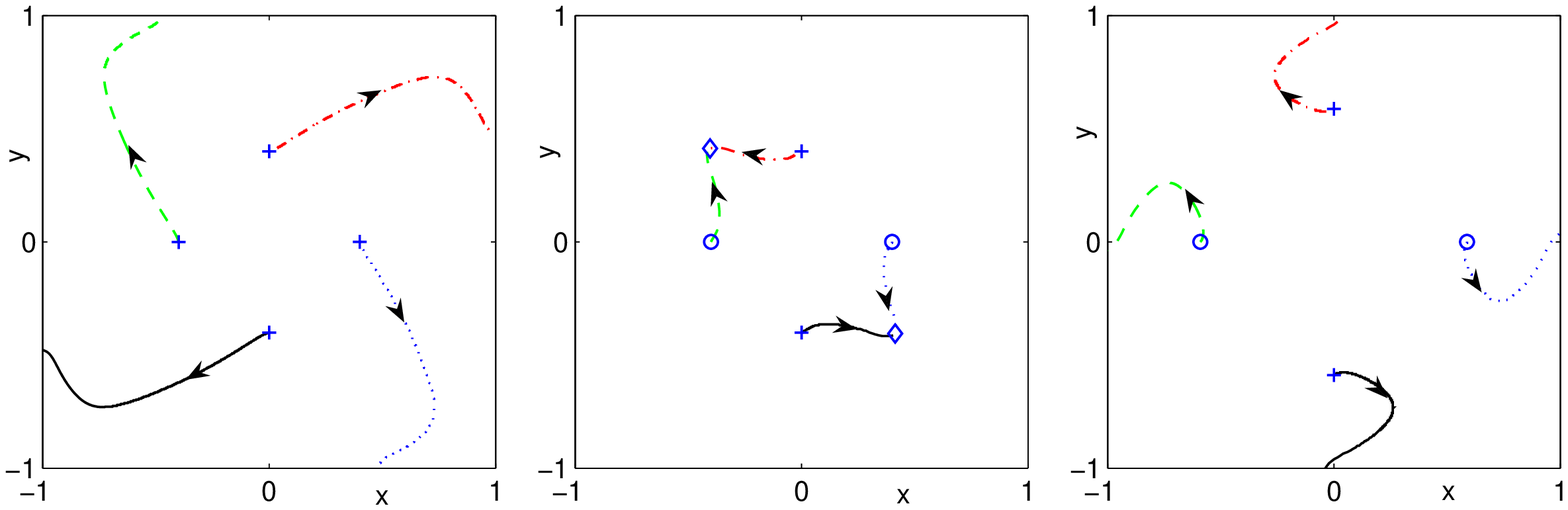,height=4cm,width=12.5cm,angle=0}}
\centerline{\psfig{figure=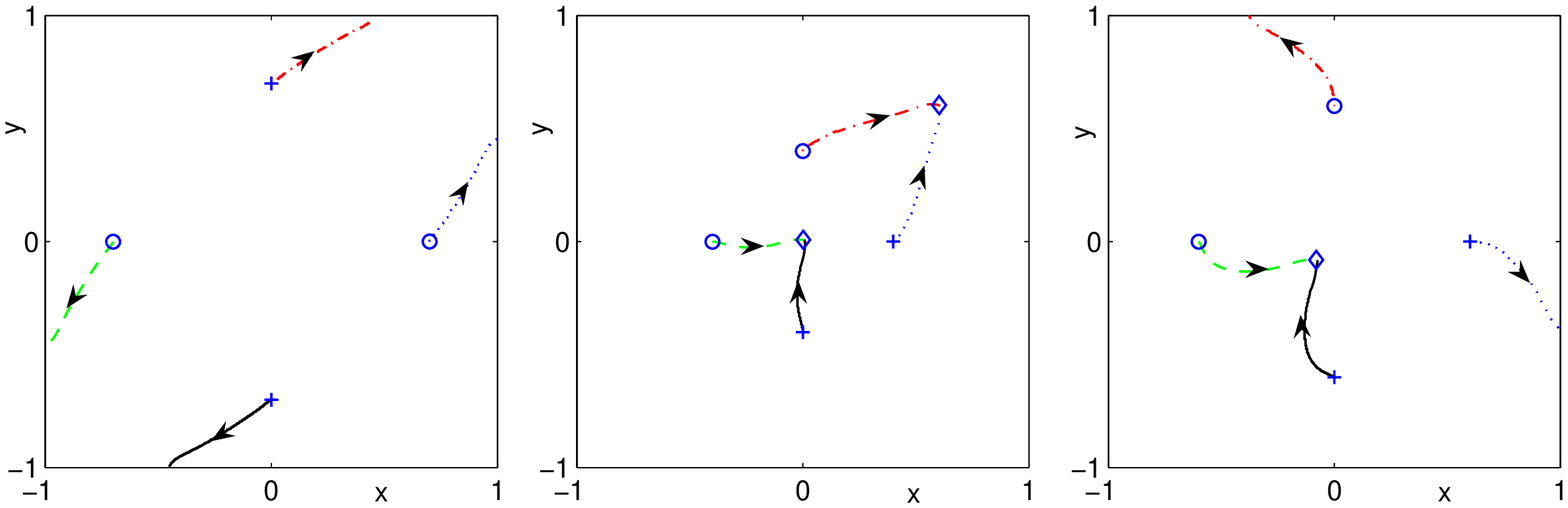,height=4cm,width=12.5cm,angle=0}}
\centerline{\psfig{figure=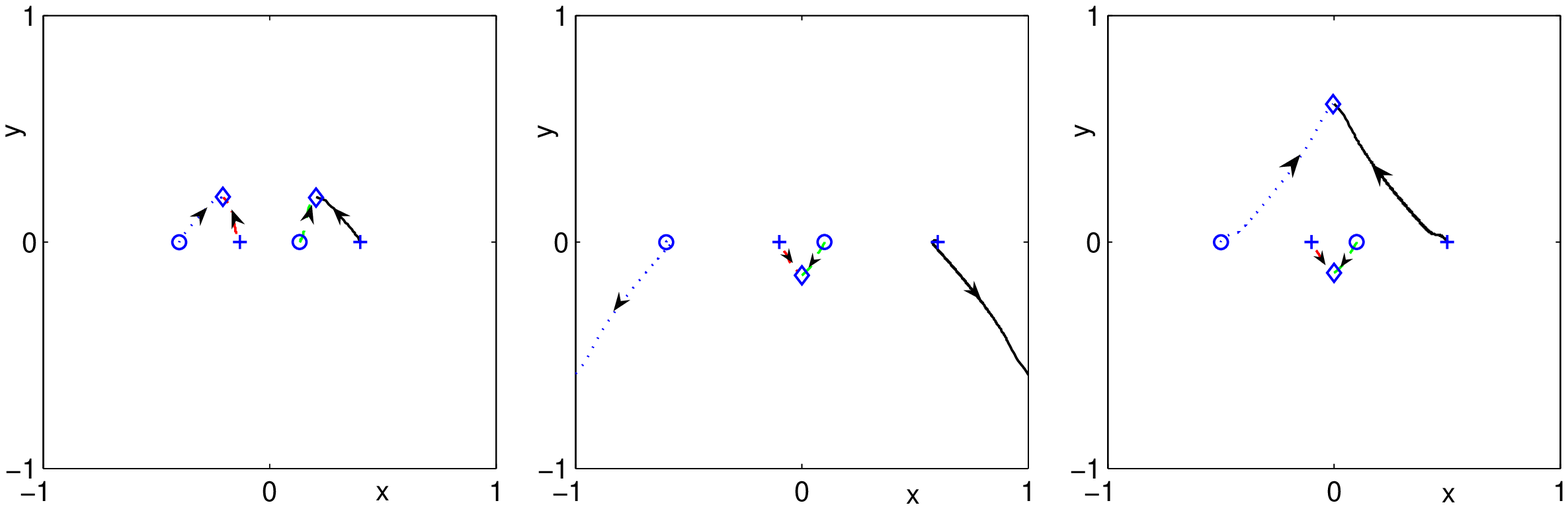,height=4cm,width=12.5cm,angle=0}}
 \caption{ Trajectory of vortex centers  for the interaction of
different vortex lattices in CGLE under  Neumman BC
with $\vep=\frac{1}{32}$  for  cases I-IX (from left to right and then from top to bottom)
in section \ref{sec: cgle_NRfVLN}. }
\label{fig: cgle-vortex-lattice-N}
 \end{figure}

\subsection{Vortex lattice}
  \label{sec: cgle_NRfVLN}

Here we present numerical results of the interaction of vortex lattices under
the CGLE dynamics with Neumann BC. We consider the following 15 cases:

case I. $M=3$, $n_1=n_2=n_3=1$,
		 $\bx_1^0=(0.4, 0)$,
         $\bx_2^0=(-0.2,\frac{\sqrt{3}}{5})$,
         $\bx_3^0=(-0.2,-\frac{\sqrt{3}}{5})$;
case II. $M=3$,  $n_1=n_2=n_3=1$, $\bx_1^0=(-0.4,0.2)$,
		 $\bx_2^0=(0,0.2)$, $\bx_3^0=(0.4,0.2)$;
case III. $M=3$, $n_1=n_2=n_3=1$, $\bx_1^0=(-0.4,0)$,
		 $\bx_2^0=(0,0)$, $\bx_3^0=(0.4,0)$; 		
case IV. $M=3$, $-n_1=n_2=n_3=1$,
         $\bx_1^0=(0.4, 0)$;
		$\bx_2^0=(-0.2,\frac{\sqrt{3}}{5})$,
         $\bx_3^0=(-0.2,-\frac{\sqrt{3}}{5})$;
case V. $M=3$, $-n_2=n_1=n_3=1$,
		$\bx_1^0=(-0.4,0)$,
		$\bx_2^0=(0,0)$,
		$\bx_3^0=(0.4,0)$;
case VI. $M=3$, $-n_2=n_1=n_3=1$,
		$\bx_1^0=(-0.7,0)$,
		$\bx_2^0=(0,0)$,
		$\bx_3^0=(0.7,0)$;
case VII. $M=4$, $n_1=n_2=n_3=n_4=1$,
		  $\bx_1^0=(0.4,0)$,
		  $\bx_2^0=(0,0.4)$,
		  $\bx_3^0=(-0.4,0)$,
		  $\bx_4^0=(0,-0.4)$;
case VIII. $M=4$, $n_1=n_3=-1$, $n_2=n_4=1$,
		  $\bx_1^0=(0.4,0)$,
		  $\bx_2^0=(0,0.4)$,
		  $\bx_3^0=(-0.4,0)$,
		  $\bx_4^0=(0,-0.4)$;		  		
case IX.  $M=4$, $n_1=n_3=-1$, $n_2=n_4=1$,
		  $\bx_1^0=(0.59,0)$,
		  $\bx_2^0=(0,0.59)$,
		  $\bx_3^0=(-0.59,0)$,
		  $\bx_4^0=(0,-0.59)$;	
case X.  $M=4$, $n_1=n_3=-1$, $n_2=n_4=1$,
		  $\bx_1^0=(0.7,0)$,
		  $\bx_2^0=(0,0.7)$,
		  $\bx_3^0=(-0.7,0)$,
		  $\bx_4^0=(0,-0.7)$;		  		  	
case XI.  $M=4$, $n_2=n_3=-1$, $n_1=n_4=1$,
		  $\bx_1^0=(0.4,0)$,
		  $\bx_2^0=(0,0.4)$,
		  $\bx_3^0=(-0.4,0)$,
		  $\bx_4^0=(0,-0.4)$;  		  		  		
case XII.  $M=4$, $n_2=n_3=-1$, $n_1=n_4=1$,
		  $\bx_1^0=(0.6,0)$,
		  $\bx_2^0=(0,0.6)$,
		  $\bx_3^0=(-0.6,0)$,
		  $\bx_4^0=(0,-0.6)$;  		  		
case XIII.  $M=4$, $n_1=n_3=-1$, $n_2=n_4=1$,
		  $\bx_1^0=(0.4,0)$,
		  $\bx_2^0=(-0.4/3,0)$,
		  $\bx_3^0=(0.4/3,0)$,
		  $\bx_4^0=(0.4,0)$;			  		
case XIV.  $M=4$, $n_1=n_3=-1$, $n_2=n_4=1$,
		  $\bx_1^0=(0.4,0)$,
		  $\bx_2^0=(-0.4/3,0)$,
		  $\bx_3^0=(0.4/3,0)$,
		  $\bx_4^0=(0.4,0)$;			  		
case XV.  $M=4$, $n_1=n_3=-1$, $n_2=n_4=1$,
		  $\bx_1^0=(-0.6,0)$,
		  $\bx_2^0=(-0.1,0)$,
		  $\bx_3^0=(0.1,0)$,
		  $\bx_4^0=(0.6,0)$.

Fig. \ref{fig: cgle-vortex-lattice-N} shows the trajectory of the vortex centers for the above 15 cases when $\vep=\frac{1}{32}$,
and Fig. \ref{fig: cgle-vortex-lattice-dens-steady-N} depicts the contour plots of $|\psi^\vep|$ for the initial data
and corresponding steady states for  cases I, III, V, VI, VII and XIV.
From Figs. \ref{fig: cgle-vortex-lattice-N} and \ref{fig: cgle-vortex-lattice-dens-steady-N}
and ample numerical experiments (not shown here for brevity),
we can make the following observations:
(i). The dynamics and interaction of vortex lattices under the CGLE dynamics with Dirichlet BC
depends on its initial alignment of the lattice, geometry of the domain $\mathcal{D}$.
(ii). For a lattice of $M$ vortices,
if they have the same index, then at least $M-1$ vortices will move out of the domain at finite time
and no collision will happen at any time;
On the other hand, if they have opposite index, collision will happen at finite time.
After collisions, the leftover vortices will continue to move and at most one vortex may be left in
the domain.
When $t$ is sufficiently large, in most cases, no vortex
can be left in the domain; but when the geometry and initial setup are properly set to be symmetric and $M$ is odd,
there maybe one vortex left in the domain.
(iii). If finally no vortex can be left in the domain, the GL functionals will always vanish as $t\rightarrow\infty$,
which indicates that the final steady state always admits the form of $\phi^{\vep}(\bx)=e^{ic_0}$ for $\bx\in\mathcal{D}$ with
$c_0$ a real constant.

\begin{figure}[t!]
\centerline{(a)\psfig{figure=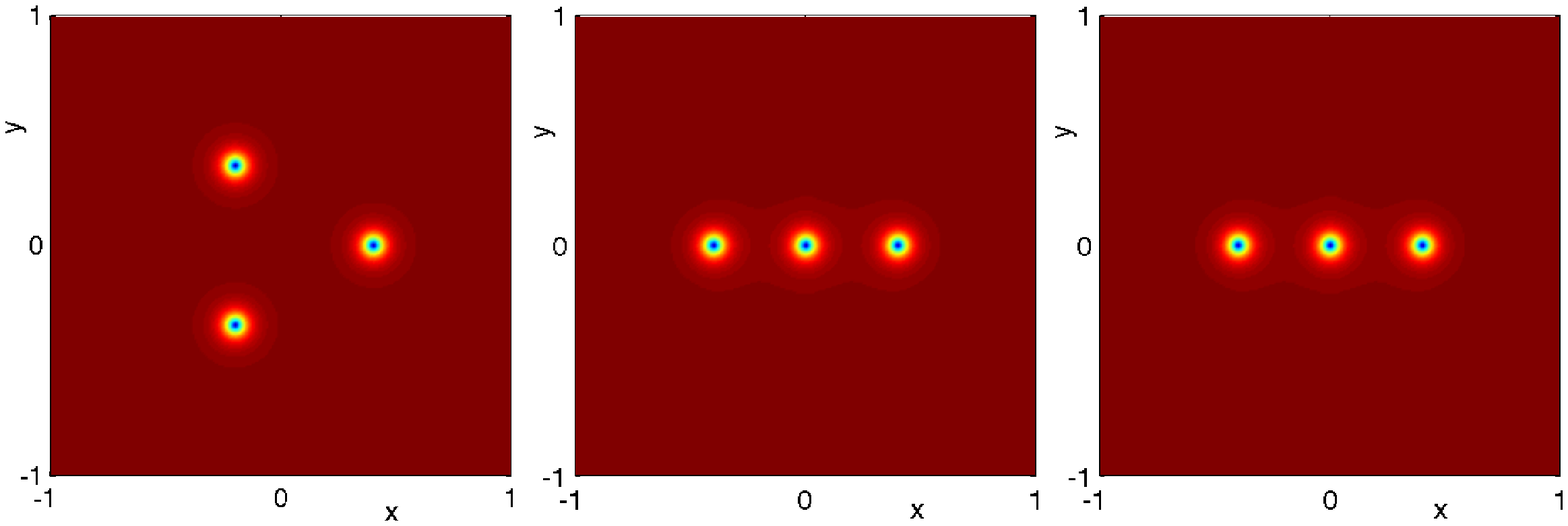,height=4cm,width=12.5cm,angle=0}}
\centerline{(b)\psfig{figure=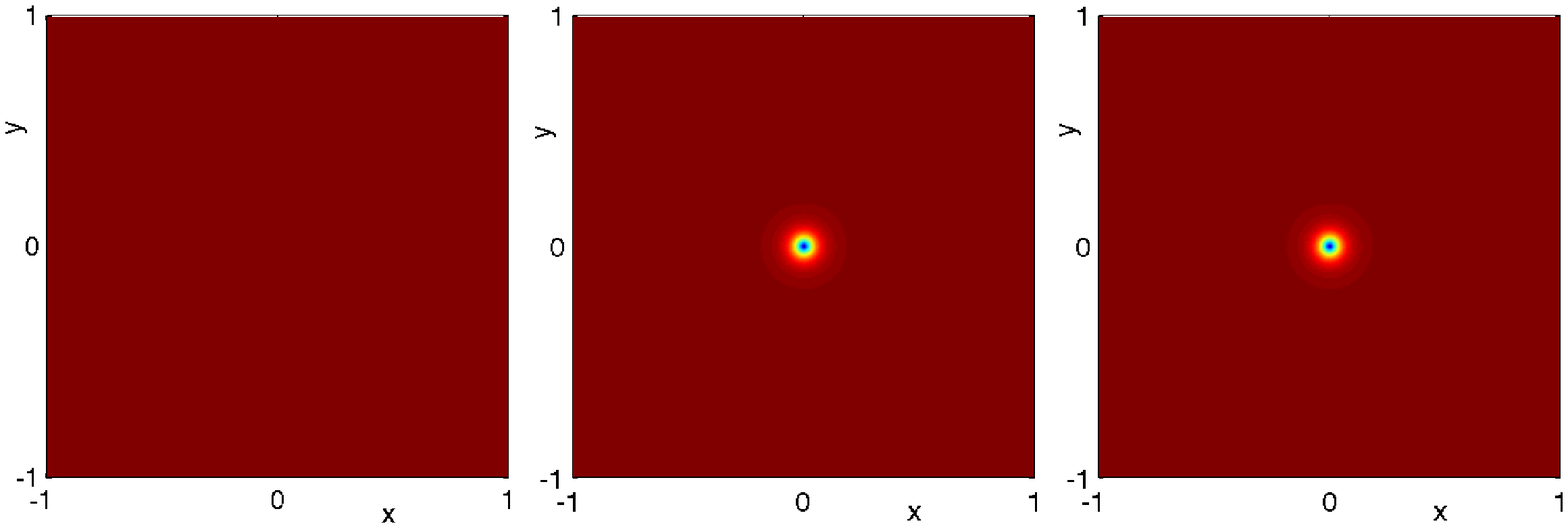,height=4cm,width=12.5cm,angle=0}}
\centerline{(c)\psfig{figure=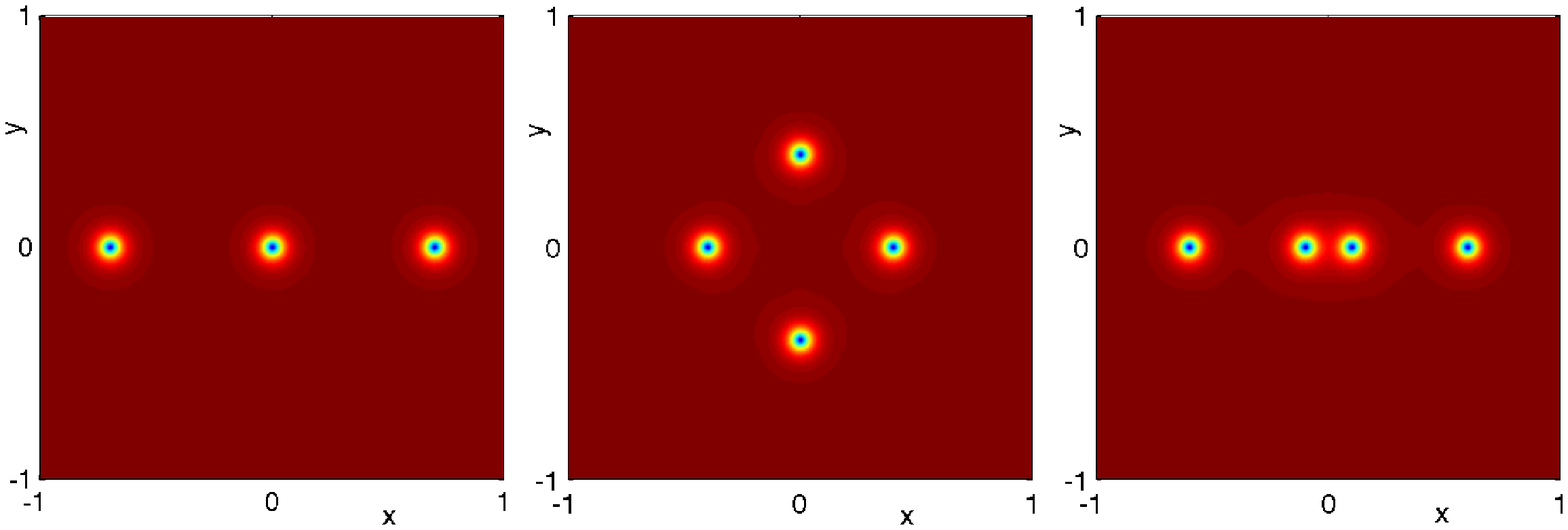,height=4cm,width=12.5cm,angle=0}}
\centerline{(d)\psfig{figure=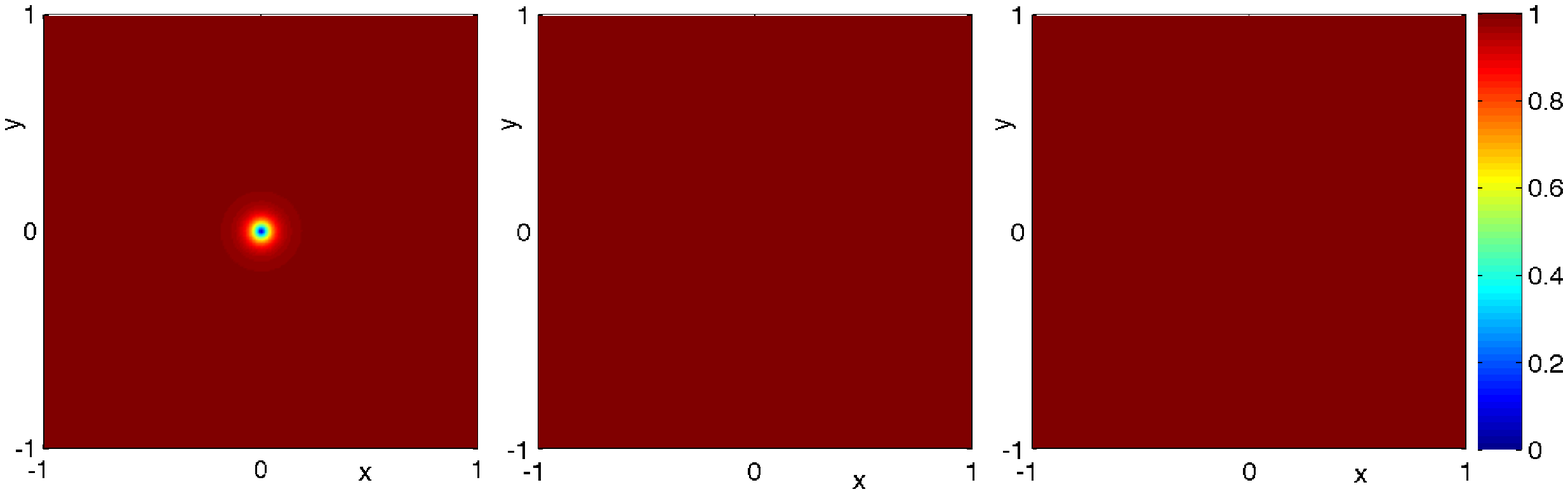,height=4cm,width=12.5cm,angle=0}}
 \caption{ Contour plots of $|\psi^\vep(\bx,t)|$ for the initial data
 ((a) \& (c)) and corresponding steady states ((b) \& (d)) of vortex
 lattice  in CGLE dynamics under  Neumman BC
with $\vep=\frac{1}{32}$ and for  cases I, III, V, VI, VII and XIV (from left to right and then from top to bottom)
in section \ref{sec: cgle_NRfVLN}. }
\label{fig: cgle-vortex-lattice-dens-steady-N}
 \end{figure}

\section{Conclusion}\label{sec: cgle_con}

In this paper, we proposed efficient and accurate numerical methods
to simulate  complex Ginzburg-Landau
equation (CGLE) with a dimensionless parameter $0<\vep<1$ on
bounded domains with either Dirichlet or homogenous Neumann
BC and its corresponding reduced dynamical laws (RDLs). By these numerical methods,
we studied numerically vortex dynamics and interaction in the
CGLE  and compared them with those obtained from the corresponding RDLs under
different initial setups. To some extent, we found that vortex dynamics in the CGLE is
a hybrid of that in GLE and NLSE, which can be reflected from the fact that CGLE is a
combination equation between GLE and NLSE.

Based on our extensive numerical
results, we verified that the dynamics of
vortex centers under the CGLE dynamics converges to
that of the RDLs when $\vep\to 0$
before they collide and/or move out of the domain.
Apparently, when the vortex center moves out of the domain,
the reduced dynamics laws are no longer valid; however, the dynamics and interaction
of quantized vortices are still physically interesting and they can be
obtained from the direct numerical
simulations for the CGLE with fixed $\vep>0$ even after they collide and/or move
out of the domain.
We also identified the parameter
regimes where the RDLs agree
with qualitatively and/or quantitatively as well as
fail to agree with those from the CGLE dynamics.
Some very interesting nonlinear phenomena related to the quantized
vortex interactions in the CGLE
were also observed from our direct numerical simulation
results of CGLE. Different steady state patterns
of vortex lattices under the CGLE dynamics were obtained
numerically. From our numerical results, we observed that both
boundary conditions and domain geometry affect
significantly on vortex dynamics and interaction, which can exhibit different
interaction patterns compared with those in the whole space case \cite{YZWBQD1,YZWBQD2}.

\section*{Acknowledgements}

The authors would like to express their sincere thanks to Dr. Dong Xuanchun for
stimulating discussions. This work was supported by the Singapore A*STAR SERC Grant No. 1224504056.
Part of this work was done when the first author was visiting IMS at NUS and the
second author was visiting IPAM at UCLA in 2012.

\end{document}